\documentclass[11pt]{amsart}
\usepackage{fullpage,xcolor,graphicx}
\usepackage[T1]{fontenc}
\usepackage{hyperref}
\usepackage{pifont}
\usepackage{amsthm}
\usepackage{amsfonts}
\usepackage{amssymb}
\usepackage[mathscr]{euscript}
\usepackage[all]{xy}
\usepackage{amsmath}
\usepackage{epsfig}
\usepackage{latexsym}
\usepackage{stackengine}
\usepackage{bbding}

\newtheorem{theorem}{Theorem}[section]
\newtheorem{lemma}[theorem]{Lemma}
\newtheorem{proposition}[theorem]{Proposition}

\newtheorem{thm}{Theorem}

\theoremstyle{definition}
\newtheorem{definition}[theorem]{Definition}
\newtheorem{example}[theorem]{Example}
\theoremstyle{remark}
\newtheorem{remark}[theorem]{Remark}

\newcommand{\lk}{\operatorname{\ell{\it k}}}

\title[Concordances to prime hyperbolic virtual knots]{Concordances to prime hyperbolic virtual knots}

\author[M. Chrisman]{Micah Chrisman}
\address{Department of Mathematics, The Ohio State University, Columbus, Ohio, 43210}
\email{chrisman.76@osu.edu}

\subjclass[2010]{Primary: 57M25, Secondary: 57M27}
\keywords{virtual knot, concordance, prime knot, satellite knot, hyperbolic knot, almost classical knot, complementary tangle.}

\begin{document}
\begin{abstract} Let $\Sigma_0,\Sigma_1$ be closed oriented surfaces. Two oriented knots $K_0 \subset \Sigma_0 \times [0,1]$ and $K_1 \subset \Sigma_1 \times [0,1]$ are said to be (virtually) concordant if there is a compact oriented $3$-manifold $W$ and a smoothly and properly embedded annulus $A$ in $W \times [0,1]$ such that $\partial W=\Sigma_1 \sqcup -\Sigma_0$ and $\partial A=K_1 \sqcup -K_0$. This notion of concordance, due to Turaev, is equivalent to concordance of virtual knots, due to Kauffman. A prime virtual knot, in the sense of Matveev, is one for which no thickened surface representative $K \subset \Sigma \times [0,1]$ admits a nontrivial decomposition along a separating vertical annulus that intersects $K$ in two points. Here we prove that every knot $K \subset \Sigma \times [0,1]$ is concordant to a prime satellite knot and a prime hyperbolic knot. For homologically trivial knots in $\Sigma \times [0,1]$, we prove this can be done so that the Alexander polynomial is preserved. This generalizes the corresponding results for classical knot concordance, due to Bleiler, Kirby-Lickorish, Livingston, Myers, Nakanishi, and Soma. The new challenge for virtual knots lies in proving primeness. Contrary to the classical case, not every hyperbolic knot in $\Sigma \times [0,1]$ is prime and not every composite knot is a satellite. Our results are obtained using a generalization of tangles in $3$-balls we call complementary tangles. Properties of complementary tangles are studied in detail.

\end{abstract}
\maketitle

\section{Introduction} 

\subsection{Motivation}\label{sec_motivate} Let $I=[0,1]$. Two oriented knots $K_0,K_1$ in the $3$-sphere are said to be \emph{concordant} if there is a smoothly and properly embedded annulus $A \subset S^3 \times I$ such that $K_i \subset S^3 \times \{i\}$ for $i=0,1$ and $\partial A=K_1 \sqcup -K_0$. Here $-K$ is $K$ with the opposite orientation. A knot in $S^3$ is said to be \emph{prime} (or \emph{locally trivial}) if every $2$-sphere that intersects $K$ transversely in two points bounds a $3$-ball that intersects $K$ in an unknotted arc. Kirby-Lickorish  \cite{kirby_lick} showed that every knot is concordant to a prime knot. Livingston  \cite{livingston} proved that every knot is concordant to a prime satellite knot. One may even choose the prime knot so that the Alexander polynomial is preserved (Bleiler \cite{bleiler}, Nakanishi \cite{nakanishi_thesis}). Myers \cite{myers} further showed that every knot in $S^3$ is concordant to a knot that is not only prime but hyperbolic. The same result holds true for links in $S^3$ (see also Soma \cite{soma}).

In \cite{turaev_cobordism}, Turaev introduced a new notion of concordance for knots in thickened surfaces $\Sigma \times I$, where $\Sigma$ is closed and oriented. Two knots $K_0\subset \Sigma_0 \times I$, $K_1 \subset \Sigma_1 \times I$ are said to be \emph{(virtually) concordant} if there is a compact oriented $3$-manifold $W$ and a smoothly and properly embedded annulus $A \subset W \times I$ such that $\partial W=\Sigma_1 \sqcup -\Sigma_0$ and $\partial A=K_1 \sqcup -K_0$. This definition coincides with concordance of virtual knots, as introduced by Kauffman \cite{lou_cob}. In fact, a virtual knot is an equivalence class of knots in thickened surfaces, where equivalence is defined by three relations: ambient isotopy, orientation preserving diffeomorphisms of surfaces, and stabilization/destabilization (see Section \ref{sec_review_virtual}). Recently, Boden-Nagel \cite{boden_nagel} proved that two classical knots in $S^3$ are concordant if and only if they are concordant as virtual knots. Since concordance classes of classical knots embed into concordance classes of virtual knots, it is natural to ask if every virtual concordance class also contains prime, hyperbolic, and satellite representatives. 

As will be seen here, the main issue turns out to be proving primeness for virtual knots. Prime virtual knots were first defined and studied by Matveev \cite{korablev_matveev,matveev_roots}. Knots $K$ in thickened surfaces $\Sigma \times I$ can be decomposed not only along $2$-spheres as in the classical case, but also along separating vertical annuli of the form $\sigma \times I$. Here $\sigma \subset \Sigma$ is a separating simple closed curve and the annulus $\sigma \times I$ intersects $K$ transversely in two points. The vertical annuli decompose both the knot and the ambient space. A virtual knot is prime if none of its representatives in thickened surfaces admits a nontrivial decomposition (see Definition \ref{defn_prime} ahead). Prime classical knots are prime as virtual knots. Every non-classical prime virtual knot is locally trivial in every thickened surface representative. However, not every locally trivial knot in a thickened surface is prime as a virtual knot (see e.g Example \ref{example_hyp_not_prime}). This fact makes it generally more difficult to prove a virtual knot is prime.
 
There has been much recent work on hyperbolic knots in thickened surfaces. Adams et al. \cite{adams}, for example, proved that fully alternating locally trivial knots in thickened surfaces have hyperbolic exterior (see also Champanerkar-Kofman-Purcell \cite{abhijit_ilya_purcell} and Howie-Purcell \cite{howie_purcell}). In fact, using the results of \cite{adams,abhijit_ilya_purcell,howie_purcell}, it is possible to construct hyperbolic knots in thickened surfaces that are not prime as virtual knots (see Example \ref{example_hyp_not_prime} ahead). Thus, hyperbolicity of a knot in a thickened surface cannot be used to prove primeness for virtual knots.

Satellite virtual knots also behave differently than might at first be expected from comparison with the classical case. Given a representative $K \subset \Sigma \times I$ of a virtual knot $\upsilon$ and a regular neighborhood $N$ of $K$, a satellite of $K$ is any knot in the interior of $N$ that is not contained in a $3$-ball in $N$ (Silver-Williams \cite{sil_will_sat}). For classical knots, every connected sum can be considered as a satellite. However, not every connected sum of virtual knots will in general be given by a satellite construction (see Example \ref{ex_connect_not_sat} ahead). Determining when a satellite virtual knot is prime also poses a new challenge.

Furthermore, even though classical knot concordance embeds into virtual knot concordance, the virtual concordance classes themselves exhibit unusual and unexpected behavior. We give two illustrations. First, a knot $K \subset \Sigma \times I$ is said to be \emph{(virtually) topologically slice} if it bounds a locally flat disc in some thickened $3$-manifold $W \times I$, where $W$ is compact and oriented, and $\partial W=\Sigma$. As is well known, a classical knot with unit Alexander polynomial is topologically slice (Freedman-Quinn \cite{freedman_quinn}, Theorem 11.7B).  However, the non-classical virtual knot 5.2080 (from Green's table \cite{green}) has unit Alexander polynomial, but is not topologically slice (see \cite{bcg2}, Tables 1 and 2). The second illustration is the band pass class of a knot in $S^3$. Kauffman \cite{on_knots} proved that concordant knots in $S^3$ have the same band pass class. However, every concordance class of long virtual knots contains band pass inequivalent representatives \cite{band_pass}. With these features of the virtual landscape in mind, we proceed to the statement of our main results.

\subsection{Main Results} \label{sec_summary} We will work here exclusively in the smooth category. Informally, a virtual knot will be called \emph{hyperbolic} if it can be represented by a knot in some thickened closed oriented surface $\Sigma \times I$ having hyperbolic exterior. A more precise definition is given in Section \ref{sec_hyp_defn}. A virtual knot is called \emph{almost classical (AC)} if it can be represented by a homologically trivial knot in some thickened surface. The \emph{virtual genus} of a virtual knot is the minimum genus surface on which it can be represented. The main results in this paper are the following three theorems. 

\begin{thm} \label{thm_A} Every virtual knot is concordant to a prime satellite virtual knot having the same virtual genus.
\end{thm}

\begin{thm} \label{thm_B} Every virtual knot is concordant to a prime hyperbolic virtual knot having the same virtual genus.
\end{thm}

\begin{thm} \label{thm_C} Every almost classical (AC) knot is concordant to a prime satellite AC knot and a prime hyperbolic AC knot, both of which preserve the Alexander polynomial and the virtual genus.
\end{thm}

The three theorems above can be interpreted in several different ways. From one point of view, the theorems above add a dimension of knot diagrammatics to the surprising convergence of algebra, geometry, and topology that occurs in the case of classical knots. Alternatively, many combinatorially defined invariants of virtual knots are now known to be concordance invariants \cite{bbc,bcg2,bcg1}. Some examples are the odd writhe, writhe polynomial, and graded genus. Our three main theorems demonstrate that there are interesting geometric properties of virtual knots that cannot be detected by any concordance invariant.

Here we briefly outline the ideas used in the proofs of Theorems \ref{thm_A}, \ref{thm_B} and \ref{thm_C}. As the central obstacle throughout lies in proving primeness, this difficulty is tackled first. Section \ref{sec_compress} presents some elementary theorems about local triviality for knots in $\Sigma \times I$ and primeness for virtual knots. These are then applied to proving Theorems \ref{thm_A}, \ref{thm_B}, and \ref{thm_C} in Sections \ref{sec_sat}, \ref{sec_hyp}, and \ref{sec_alex}, respectively. 

Theorem \ref{thm_A} is proved by generalizing Livingtson's theorem to virtual knots. Theorem \ref{thm_B} is proved using a generalization of tangles in $3$-balls. For classical knots, concordances to prime knots or hyperbolic knots in $S^3$ can be constructed using tangles as building blocks (see for example, Kirby-Lickorish \cite{kirby_lick}, Lickorish \cite{lick}, Myers \cite{myers}, Nakanishi \cite{nakanishi_thesis}, and Soma \cite{soma}). Here we introduce \emph{complementary tangles} in $\Sigma \times I$. A (2-string) complementary tangle is a pair of disjoint arcs in the interior of $\Sigma \times I$, where the interior of an embedded $3$-ball $B$ has been removed from the interior of $\Sigma \times I$. The ends of the arcs of a complementary tangle are in $\partial B$. Knots with the desired properties in $\Sigma \times I$ can be constructed by gluing tangles to complementary tangles. In addition to proving Theorem \ref{thm_B}, Section \ref{sec_hyp} explores the basic properties of complementary tangles.

We restrict to almost classical knots in Theorem \ref{thm_C} as these are exactly the collection of virtual knots for which the Alexander polynomial has its usual definition in terms of Seifert surfaces (Boden et al. \cite{acpaper}). As in the classical case, the AC Alexander polynomial satisfies a skein relation. Using this skein relation, we show how to arrange the geometric constructions in the proofs of Theorems \ref{thm_A} and \ref{thm_B} to produce AC knots having the same Alexander polynomial. The remainder of this section gives background material on virtual knots, virtual knot concordance, and prime virtual knots.

\subsection{Virtual knots} \label{sec_review_virtual} In this subsection, we recall the relationship between virtual knots and knots in thickened surfaces. First, a \emph{virtual knot diagram} is an immersed circle in $\mathbb{R}^2$, with only transversal self-intersections, where each double point is marked as either a usual over/under crossing or as virtual crossing. A virtual crossing is denoted with a small circle centered at the double point. Two virtual knot diagrams $\upsilon_1,\upsilon_2$ are said to be equivalent, denoted $\upsilon_1 \leftrightharpoons \upsilon_2$, if they may be obtained from one another by a finite sequence of extended Reidemeister moves \cite{KaV}. See Figure \ref{fig_rmoves}. A virtual knot is an equivalence class of virtual knot diagrams. Virtual links are defined analogously as multi-component diagrams in the plane with both classical and virtual crossings allowed. A virtual knot that has a diagram with no virtual crossings is called a \emph{classical knot}. The set of classical knots embeds into the set of virtual knots \cite{GPV}. In other words, if two classical knots are equivalent as virtual knots, then they are equivalent as knots in the $3$-sphere.
\newline

\begin{figure}[htb]
\begin{tabular}{c}  \def\svgwidth{5.5in}
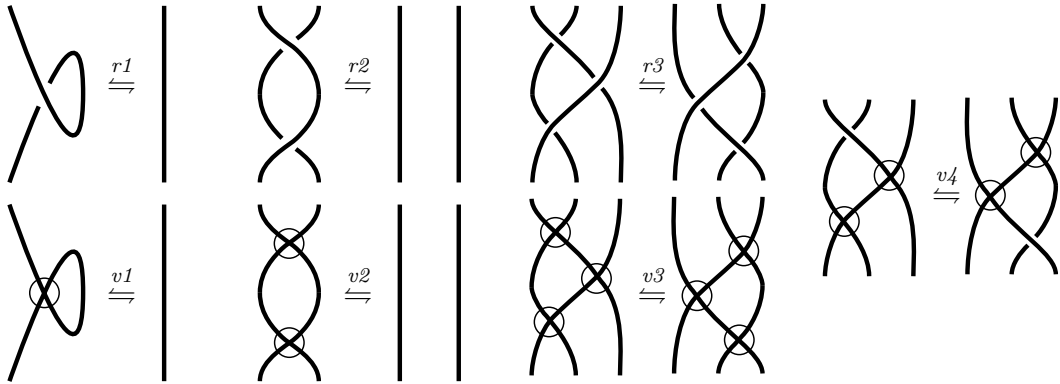 \end{tabular} 
\caption{The extended Reidemeister moves.} \label{fig_rmoves}
\end{figure}

Given a virtual link diagram, a link in a thickened surface is obtained via the \emph{Carter surface algorithm} \cite{carter}. The algorithm constructs the supporting surface via a handlebody decomposition. A $0$-handle is centered at each classical crossing. Each arc between the classical crossings corresponds to the core of a $1$-handle with the blackboard framing. Virtual crossings are ignored, so that at a virtual crossing we have $1$-handles that pass over and under one another.  The union of the $0$- and $1$-handles is a compact surface with boundary. A closed oriented surface $\Sigma$ is formed by attaching $2$-handles to all of its boundary components. An example is given in Figure \ref{fig_carter}. A link diagram of a knot in $\Sigma\times I$ appears as the union of the classical crossings and the cores of the $1$-handles. For a classical knot diagram, the Carter surface is a $2$-sphere. Classical knots in $S^3$ are thus considered as knots in the thickened surface $S^2 \times I$.

\begin{figure}[htb]
\begin{tabular}{ccc} \begin{tabular}{c}\def\svgwidth{1.3in}
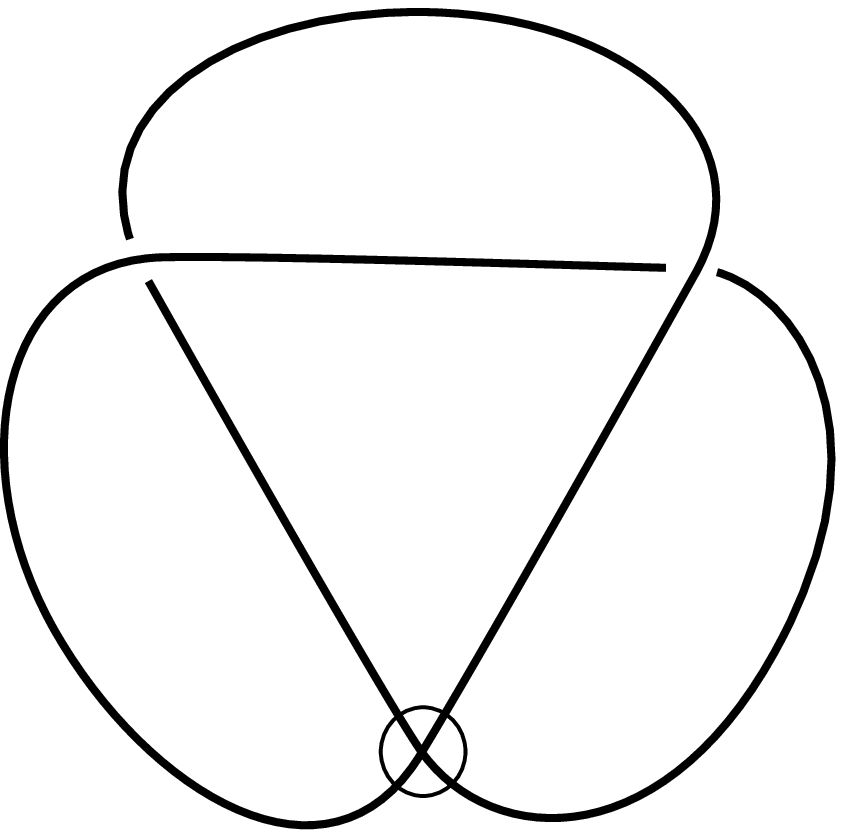 \end{tabular} &\begin{tabular}{c}\def\svgwidth{1.6in}
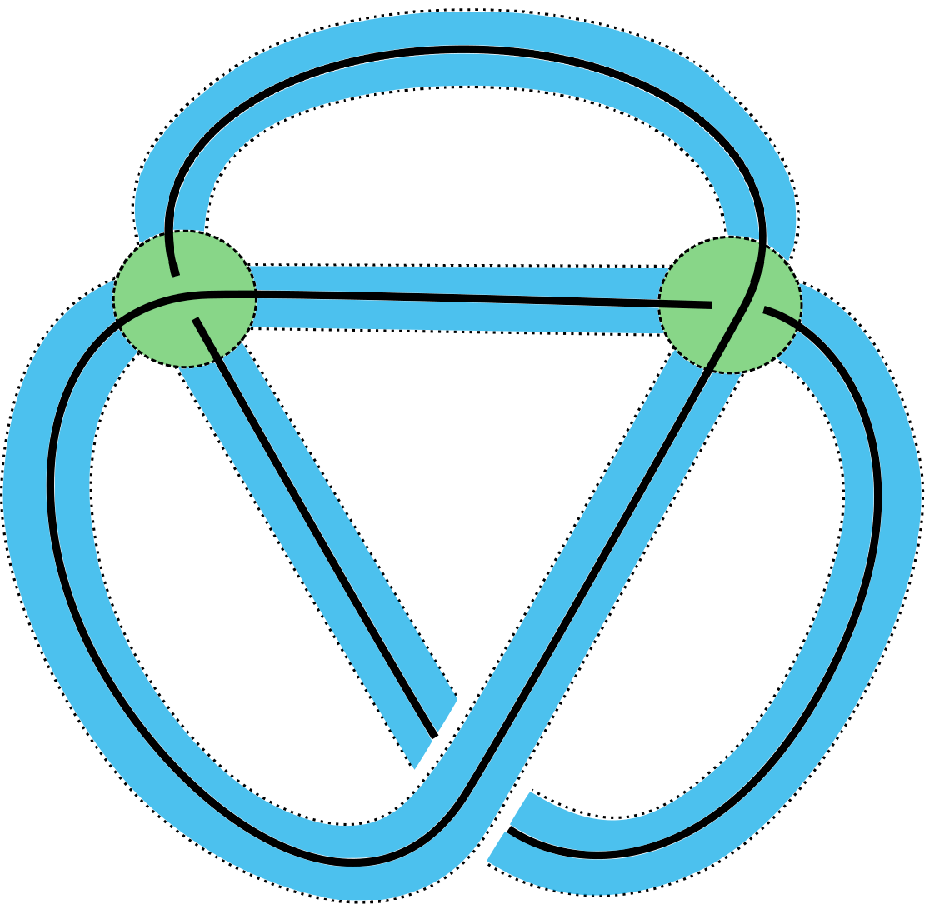 \end{tabular}
&
\begin{tabular}{c}\def\svgwidth{1.8in}
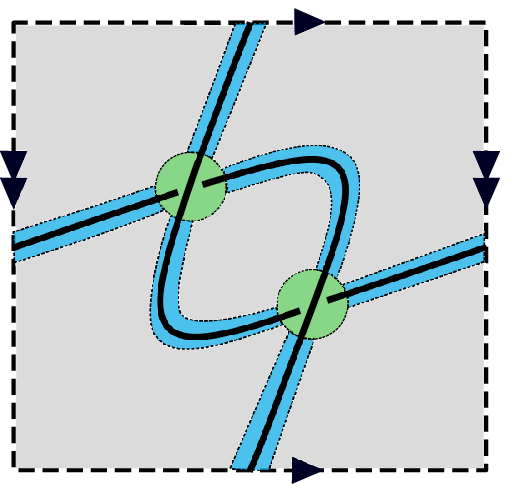 \end{tabular}
\end{tabular}
\caption{The Carter surface algorithm for a virtual trefoil diagram: attach two $0$-handles (in green), four $1$-handles (in blue), and two $2$-handles (in gray).} \label{fig_carter}
\end{figure}
 
Next we define \emph{stabilization} and \emph{destabilization}. These operations allow one to transfer between different thickened surface representations of the same virtual knot. Let $K \subset \Sigma \times I$ be a knot and suppose that $\sigma \subset \Sigma$ is a smoothly embedded closed curve such that $ K \cap (\sigma \times I)=\emptyset$. Cut $\Sigma \times I$ along $\sigma \times I$, cap off with two thickened discs, and then discard if necessary any components not containing $K$. The resulting knot $K' \subset \Sigma' \times I$ is said to be obtained from $K \subset \Sigma \times I$ by a \emph{destabilization}.  The inverse operation is called a \emph{stabilization}. See Figure \ref{fig_destab}. 

The \emph{virtual genus} of a virtual knot is the smallest genus among all closed oriented surfaces $\Sigma$ on which the virtual knot can be represented. If $K \subset \Sigma \times I$ is a representative of a virtual knot and the genus of $\Sigma$ is the virtual genus of $\upsilon$, then $K$ will be called \emph{minimal}. By Kuperberg, any representative may be destabilized to a minimal one.  

\begin{theorem}[Kuperberg \cite{kuperberg}]\label{thm_kuperberg} Given any representative $K\subset \Sigma \times I$ of a virtual knot $\upsilon$, there is a sequence of destabilizations taking $K$ to a minimal representative of $\upsilon$. Moreover, if $K \subset \Sigma \times I$ and $K' \subset \Sigma' \times I$ are two minimal representatives, then there is an orientation preserving diffeomorphism $f:\Sigma \times I \to \Sigma' \times I$ such that $f(\Sigma \times \{0\})=\Sigma' \times \{0\}$, $f(\Sigma \times \{1\})=\Sigma' \times \{1\}$, and $f(K)=K'$. 
\end{theorem} 

\begin{figure}[htb]
\[
\xymatrix{\begin{tabular}{c}\def\svgwidth{4in}
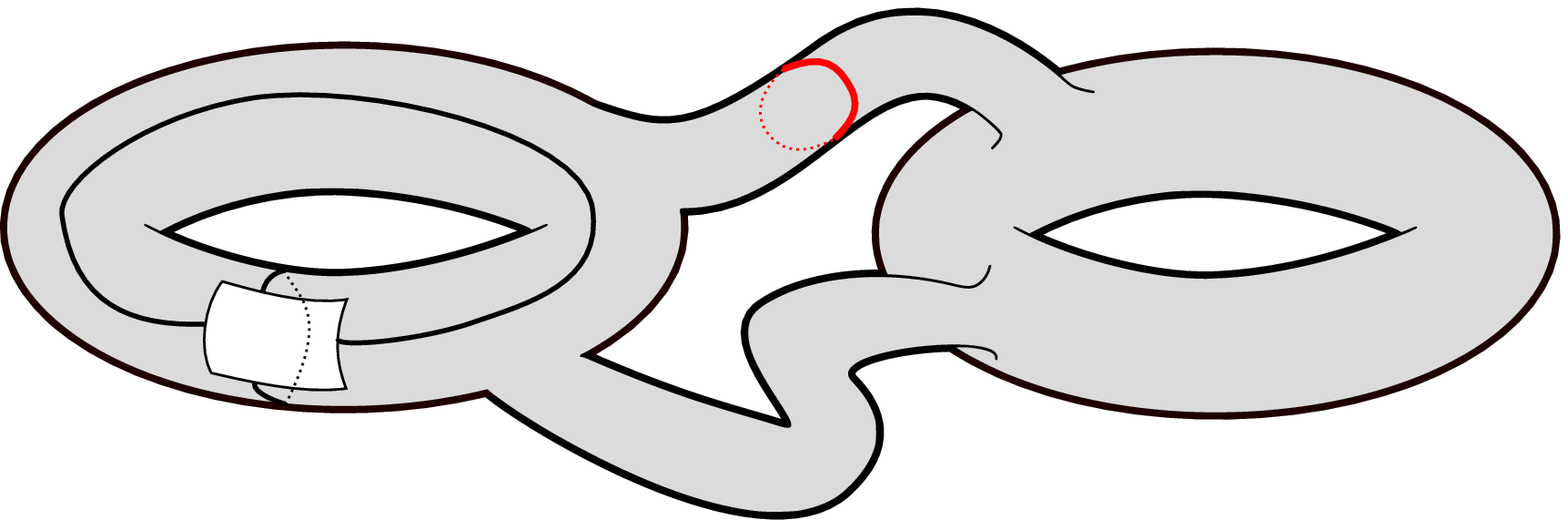 \end{tabular}  \ar@/^3pc/[d]^{\text{destabilize}} \\ \begin{tabular}{c} \def\svgwidth{4in}
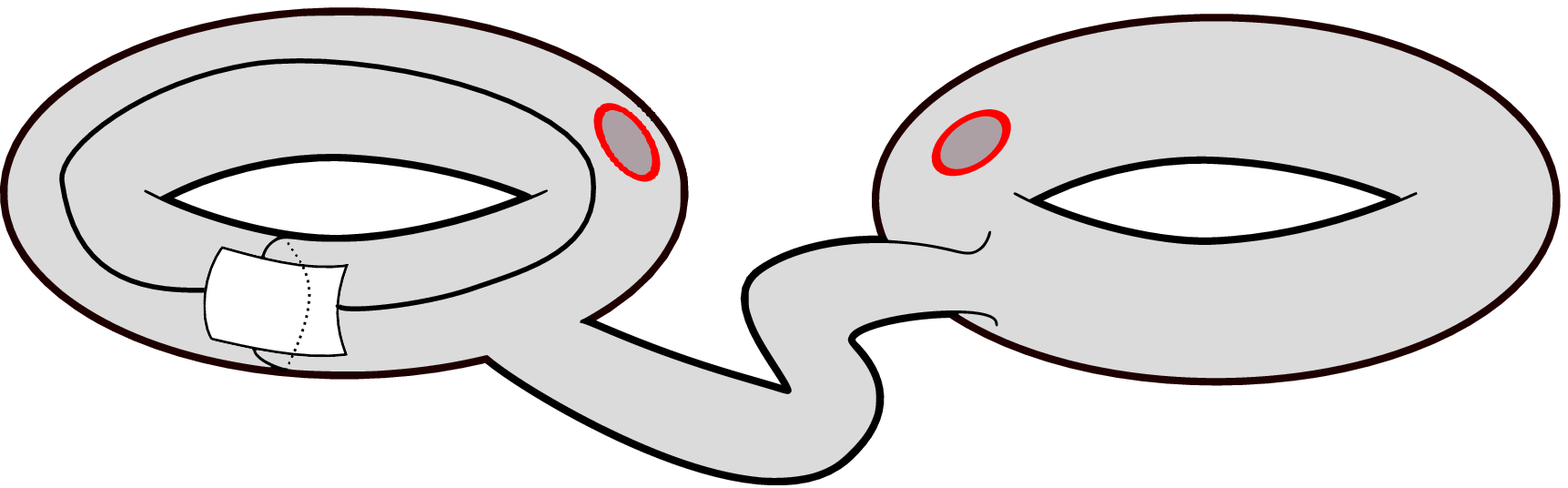\end{tabular} \ar@/^3pc/[u]^{\text{stabilize}}}
\]
\caption{Destabilization and stabilization along a vertical annulus $\sigma \times I$.} \label{fig_destab}
\end{figure}

\subsection{Virtual knot concordance} \label{sec_review_conc} Virtual knot concordance can likewise be defined either diagrammatically or topologically. The topological definition is based on Turaev's definition of concordant knots in thickened surfaces (see Section \ref{sec_motivate}).

\begin{definition}[Virtual Knot Concordance] \label{defn_v_conc} Two virtual knots $\upsilon_0, \upsilon_1$ are said to be \emph{concordant} if they have representatives $K_0 \subset \Sigma_0 \times I$ and $K_1 \subset \Sigma_1 \times I$ that are (virtually) concordant as knots in thickened surfaces. A virtual knot is \emph{(virtually) slice} if it has a representative $K \subset \Sigma \times I$ that is concordant to the unknot in $S^2 \times I$.
\end{definition}

The diagrammatic definition of virtual knot concordance is due to Kauffman \cite{lou_cob}. Figure \ref{fig_concmoves} depicts three moves on virtual link diagrams: births, deaths, and saddles. A birth adds a disjoint unknotted component to the diagram whereas a death deletes one. A saddle move is the oriented surgery on the virtual link diagram depicted on the left in Figure \ref{fig_concmoves}. The new diagram has either one less or one more component. Two virtual knot diagrams $\upsilon_0,\upsilon_1$ are concordant if and only if they are related by a finite connected\footnote{The cobordism surface that the movie represents in some thickened $3$-manifold $W \times I$ can have no closed components. In practice, one shows that the Reeb graph of the movie is connected \cite{band_pass}.} sequence of extended Reidemeister moves, births, deaths, and saddles such that:
\[
\#\text{(births)}-\#\text{(saddles)}+\#\text{(deaths)}=0.
\]
Figure \ref{fig_conc_example} shows a concordance of virtual knots consisting of two saddles and two deaths. The movie advances left-to-right and then top-to-bottom. The equivalence between the topological and diagrammatic definitions of virtual knot concordance follows from Carter-Kamada-Saito \cite{CKS}, Lemma 4.5.  

\small
\begin{figure}[htb]
\begin{tabular}{c}  \def\svgwidth{5in}
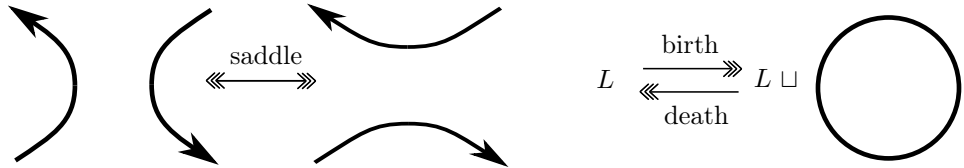 \end{tabular} 
\caption{Birth, death, and saddle moves.} \label{fig_concmoves}
\end{figure}
\normalsize

\begin{figure}[htb]
\begin{tabular}{c} \def\svgwidth{5.5in}
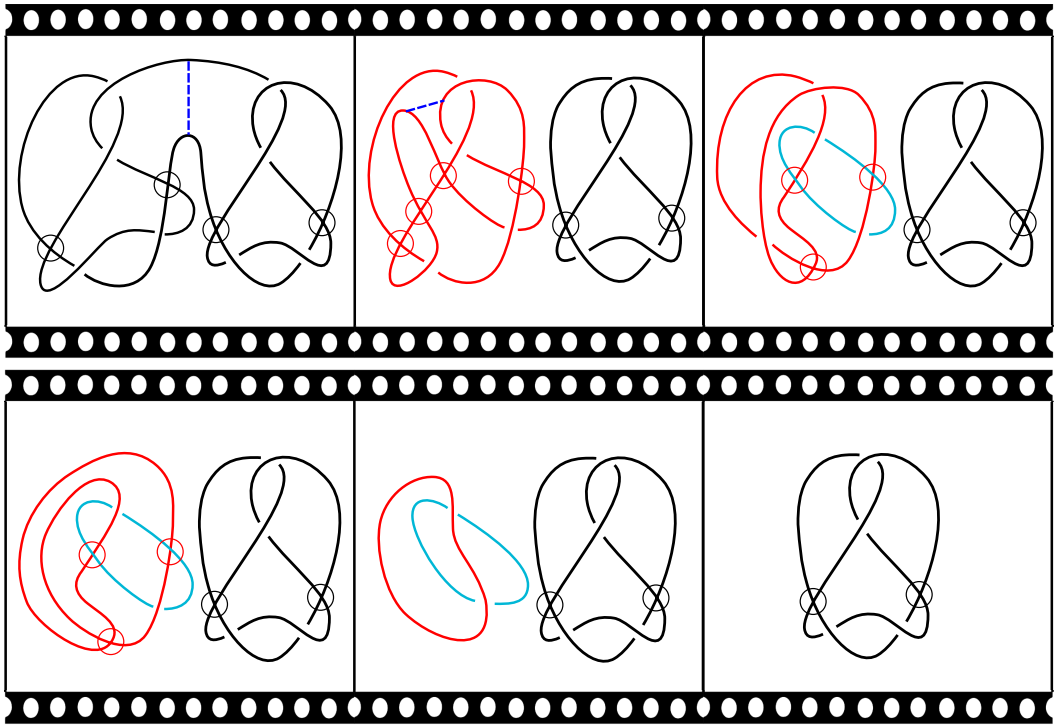 \end{tabular}
\caption{A movie of a concordance of virtual knots. The two saddles are indicated with a dotted blue line. The blue and red component are killed at the end.} \label{fig_conc_example}
\end{figure}

The same definitions carry over to virtual links. Two oriented $m$-component links $L_0 \subset \Sigma_0 \times I, L_1 \subset \Sigma_1 \times I$ are (virtually) concordant if there is a compact oriented $3$-manifold $W$ and $m$ annuli $\bigsqcup_{i=1}^m S^1 \times I$ disjointly and properly embedded in $W \times I$ such that $\partial W=\Sigma_1 \sqcup -\Sigma_0$, $\bigsqcup_{i=1}^m S^1 \times \{1\}$ is mapped to $L_1$, and $\bigsqcup_{i=1}^m S^1 \times \{0\}$ is mapped to $-L_0$. Virtual links are concordant if they have representatives that are concordant as links in thickened surfaces. Diagrammatically speaking, two $m$-component virtual link diagrams are concordant if and only if they can be obtained from one another by a finite sequence of extended Reidemeister moves, births, deaths, and saddles such that restricting to each component of the virtual link gives a concordance of virtual knots. 

\subsection{Prime virtual knots} First we give both diagrammatic and topological interpretations of connected sums. Diagrammatically, a connected sum of virtual knots can be defined as follows. For oriented diagrams $\upsilon_1$ and $\upsilon_2$, form their disjoint union $\upsilon_1 \sqcup \upsilon_2$ in the plane. A \emph{connected sum of $\upsilon_1$ and $\upsilon_2$}, denoted $\upsilon_1 \# \upsilon_2$, is a virtual knot obtained from any single saddle move between the components of the two-component virtual link $\upsilon_1 \sqcup \upsilon_2$. See Figure \ref{fig_connect_sum}. 

\tiny
\begin{figure}[htb]
\begin{tabular}{c}  \def\svgwidth{5in}
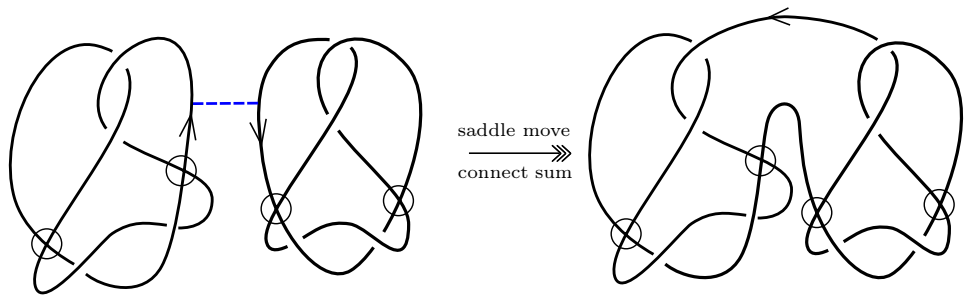 \end{tabular} 
\caption{A connected sum of virtual knots.} \label{fig_connect_sum}
\end{figure}
\normalsize

\begin{figure}[htb]
\begin{tabular}{c}  \def\svgwidth{5in}
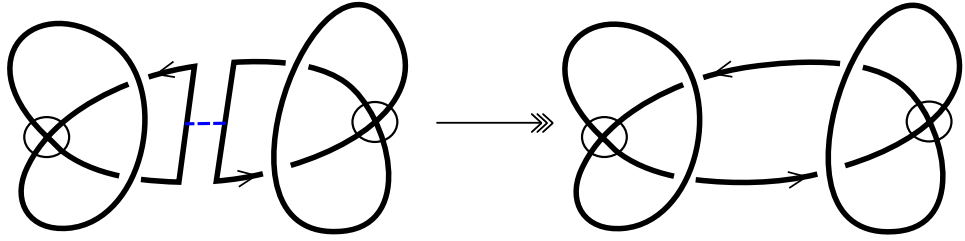 \end{tabular} 
\caption{The Kishino knot (right) as a non-trivial connected sum of two trivial knots.} \label{fig_kishino}
\end{figure}

The connected sum depends both on the initial choice of the diagrams and on the placement of the saddle. It is not in general a well-defined operation; different choices can lead to inequivalent virtual knots. Indeed, Figure \ref{fig_kishino} shows a connected sum of two trivial knots that is not a trivial knot. On the other hand, if a connected sum of virtual knots is trivial, then each summand must be trivial (see Kauffman-Manturov \cite{kauffman_manturov}, Theorem 5). Topologically, a connected sum of two virtual knots can be defined using thickened surface representatives. This leads to the definition of an \emph{annular connected sum} of knots in thickened surfaces. 

\begin{definition}[Annular connected sum and decomposition] Suppose that $K_1 \subset \Sigma_1 \times I$ and $K_2 \subset \Sigma_2 \times I$ are oriented knots.  Let $D_1 \subset \Sigma_1$ and $D_2 \subset \Sigma_2$ be $2$-discs such that for $i=1,2$, $K_i$ intersects the annulus $\partial D_i \times I$ transversely in two points and the arc $l_i=K_i \cap (D_i \times I)$ is unknotted in the $3$-ball $D_i \times I$. Let $\varphi: \partial D_1 \times I \to \partial D_2 \times I$ be an orientation reversing diffeomorphism such that $\varphi(\partial l_1)=\partial l_2$ and $\varphi(\partial D_1 \times \{i\})= \partial D_2 \times \{i\}$ for $i=0,1$. Then $((\Sigma_1 \smallsetminus \text{int}(D_1)) \times I \bigcup_{\varphi} (\Sigma_2 \smallsetminus \text{int}(D_2)) \times I$ is a thickened surface $\Sigma \times I$, where $\Sigma=\Sigma_1\#\Sigma_2$. Set $K:=\overline{K_1\smallsetminus l_1} \cup \overline{K_2 \smallsetminus l_2}$. Then $K \subset \Sigma \times I$ is called an \emph{annular connected sum of $K_1$ and $K_2$}. The inverse operation is called a \emph{annular decomposition}.
\end{definition}
\tiny
\begin{figure}[htb]
\[
\xymatrix{\begin{tabular}{c}\def\svgwidth{4in}
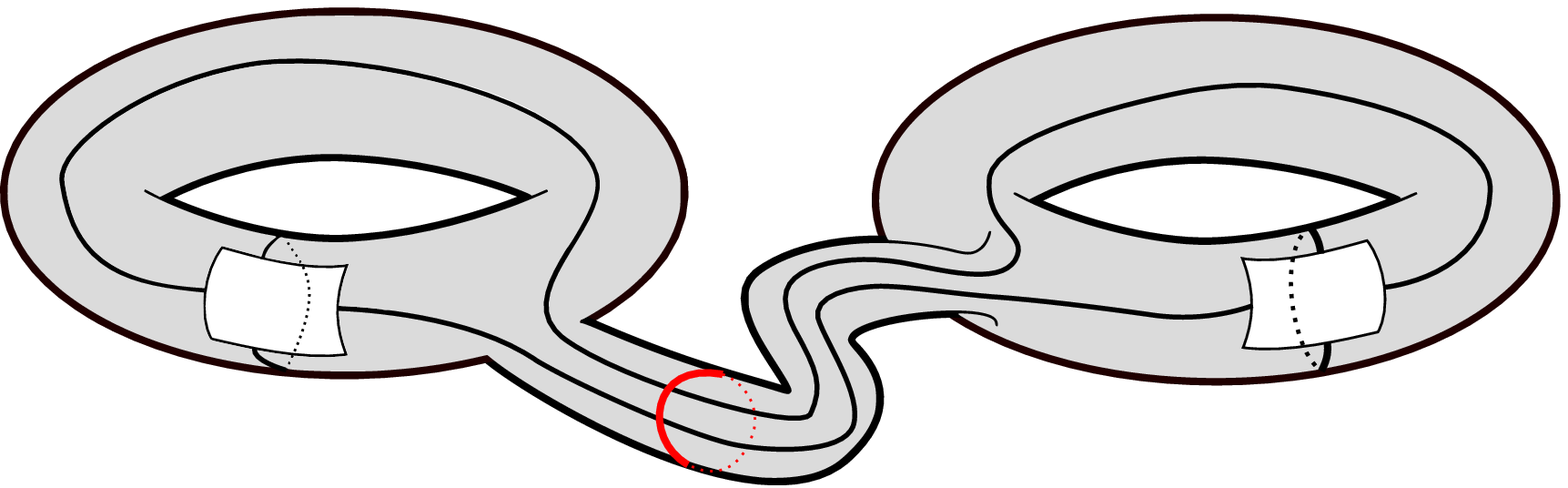 \end{tabular}  \ar@/^3pc/[d]^{\stackanchor{annular}{decomposition}} \\ \begin{tabular}{c} \def\svgwidth{4in}
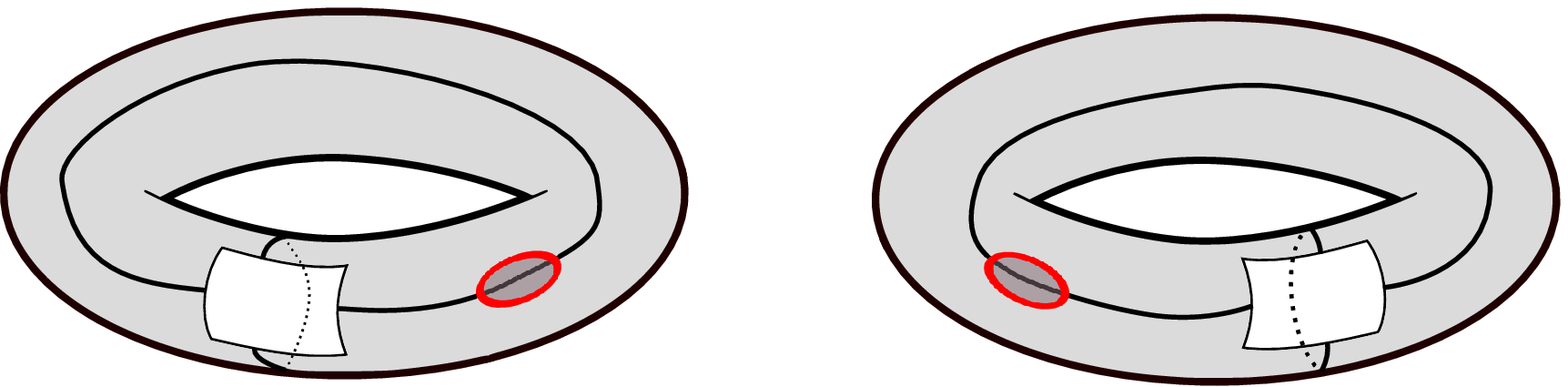\end{tabular} \ar@/^3pc/[u]^{\stackanchor{annular}{connected sum}}}
\]
\caption{An annular decomposition and connected sum.} \label{fig_annular}
\end{figure}
\normalsize

An immediate consequence of the definition is that every connected sum of virtual knots can be realized as an annular connected sum of knots in thickened surfaces and vice versa. Matveev's definition of a prime virtual knot \cite{matveev_roots}, which we give next, does not make explicit reference to either the diagrammatic or topological viewpoint. The topological definition of the connected sum will always be used in our proofs of primeness ahead.

\begin{definition}[Prime virtual knot] \label{defn_prime} A connected sum $\upsilon=\upsilon_1\#\upsilon_2$ is said to be \emph{trivial} if for some $i=1,2$, $\upsilon \leftrightharpoons \upsilon_i$ and $\upsilon_{3-i}$ is trivial. A virtual knot is said to be \emph{prime} if it is nontrivial and it does not admit nontrivial decompositions into a connected sum of virtual knots.
\end{definition}

Matveev proved that every virtual knot can be decomposed into a connected sum of prime and trivial virtual knots with uniquely determined prime summands (see \cite{matveev_roots}, Theorem 13). In \cite{akimova_matveev}, Akimova-Matveev tabulated the prime virtual knots having at most five classical crossings that have virtual genus one.

\begin{example} Figure \ref{fig_conc_example} depicts a concordance of a composite virtual knot to a prime virtual knot.  The top left of Figure \ref{fig_conc_example} shows a non-trivial connected sum of two knots. In Green's tabulation \cite{green} they are $4.99$ (left of dotted blue line) and $4.105$ (right of dotted blue line). Both of these knots have virtual genus one (see \cite{acpaper}, Table 2). They are shown to be prime in \cite{akimova_matveev}. In the notation of \cite{akimova_matveev}, $4.99=4_5$ and $4.105=4_4$ (mirror image). 
\end{example}

To prove primeness using the topological definition, it is necessary to account for the effect of destabilization on an annular decomposition. In particular, if a destabilization annulus intersects the decomposing annulus, it is not always possible to recover the annular decomposition in the destabilized knot. See Figure \ref{fig_special_example} for an example. \emph{Special connected sums and decompositions} can be used to correct for this. 

\normalsize
\begin{figure}[htb]
\[
\xymatrix{\begin{tabular}{cc}\def\svgwidth{3.2in}
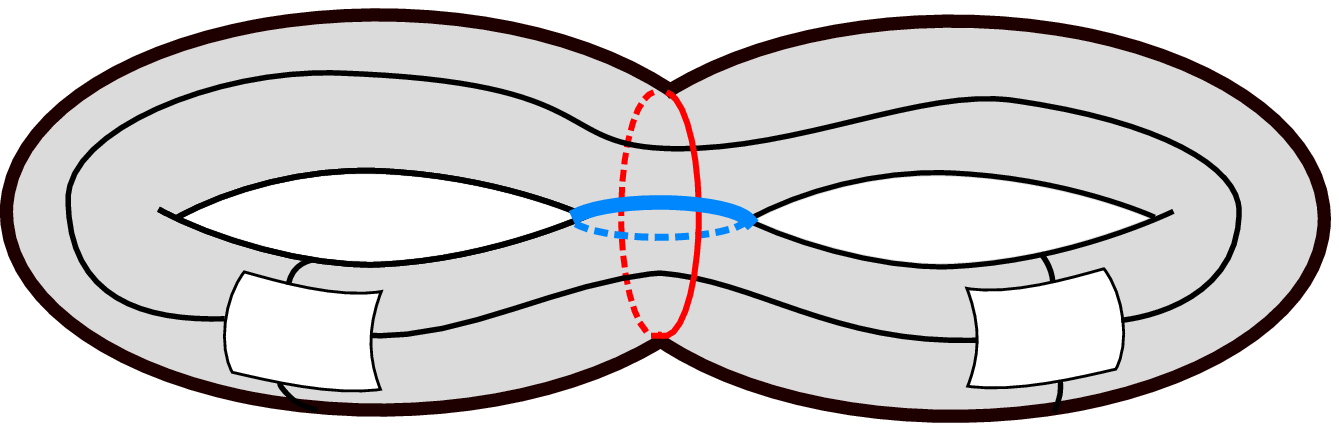 \end{tabular}  \ar@/^5pc/[r]^{\text{destabilize}} & \begin{tabular}{c} \def\svgwidth{2in}
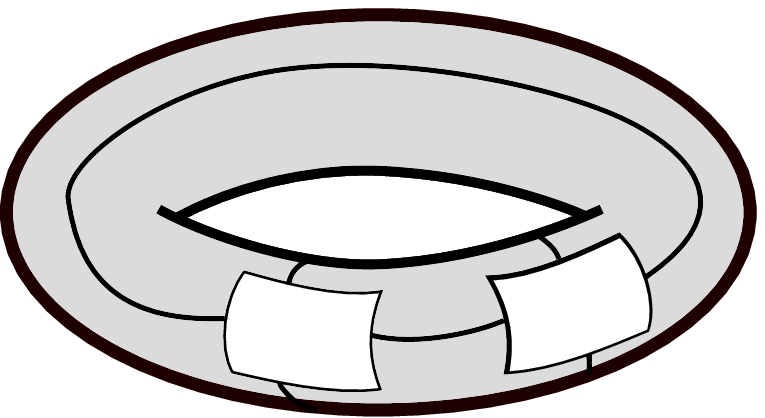\end{tabular} }
\]
\caption{The effect of destabilizing (blue curve) an annular decomposition (red curve).} \label{fig_special_example}
\end{figure}
\normalsize

\begin{definition}[Special connected sum and decomposition] \label{defn_special} For $i=1,2$, let $K_i \subset \Sigma_i \times I$ be a knot and $A_i \subset \Sigma_i \times I$ a vertical annulus intersecting $K_i$ transversely in one point $x_i$. Cutting along $A_i$ gives two copies $A_i'$, $A_i''$ of $A_i$ and two copies $x_i' \subset A_i'$ and $x_i'' \subset A_i''$ of $x_i$. Glue together $A_1' \sqcup A_1''$ to $A_2' \sqcup A_2''$ by an orientation reversing diffeomorphism $A_1' \sqcup A_1'' \to A_2' \sqcup A_2''$ so that $x_1' \sqcup x_1'' \to x_2' \sqcup x_2''$.  The top (bottom) boundary component of each annulus $A_1', A_1''$ must map to the top (respectively, bottom) boundary component of one of $A_2',A_2''$. The resulting knot is called a \emph{special connected sum of $K_1$ and $K_2$}. The inverse operation is called a \emph{special decomposition}. See Figure \ref{fig_special_defn}.
\end{definition}

\tiny
\begin{figure}[htb]
\[
\xymatrix{\begin{tabular}{c}\def\svgwidth{4in}
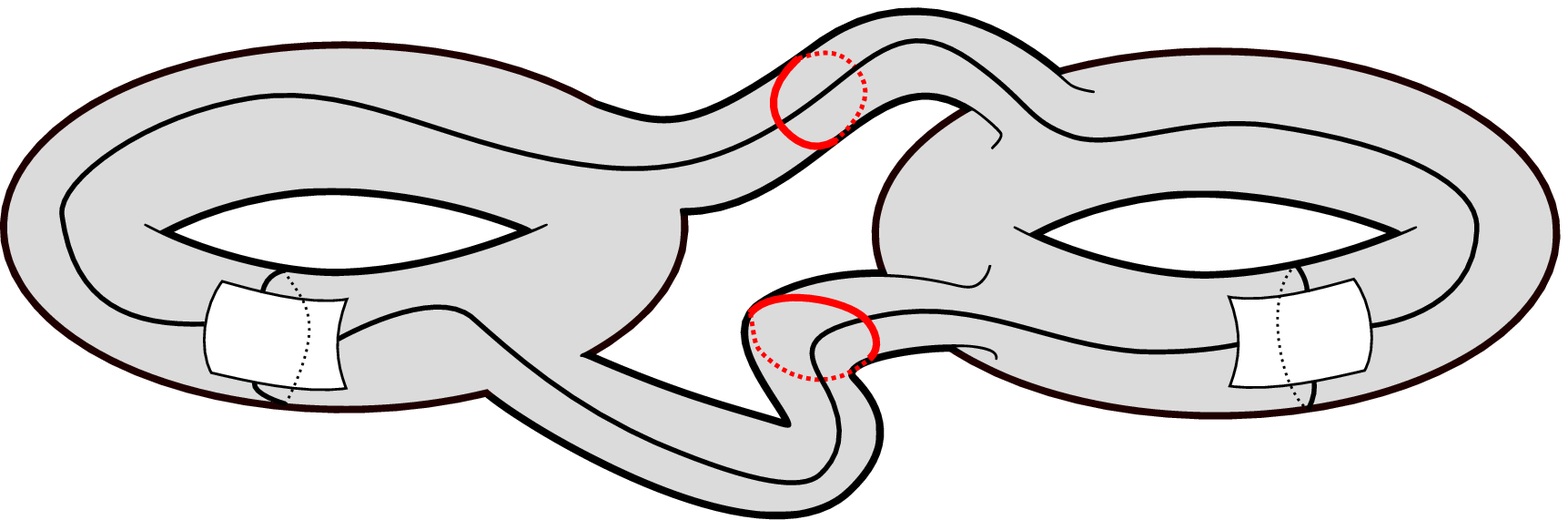 \end{tabular}  \ar@/^3pc/[d]^{\stackanchor{special}{decomposition}} \\ \begin{tabular}{c} \def\svgwidth{4in}
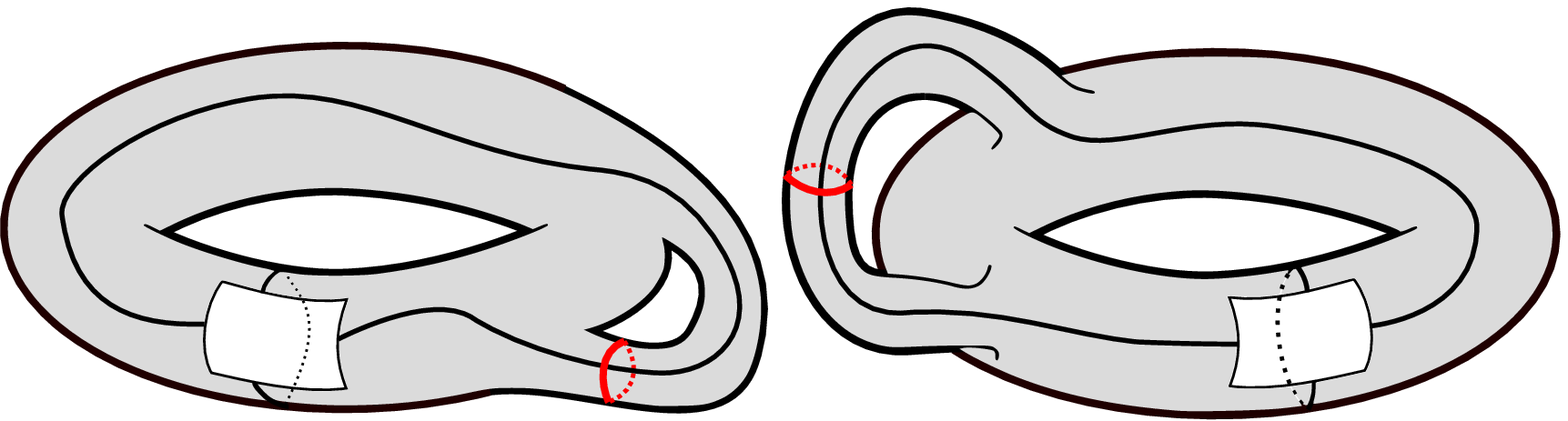\end{tabular} \ar@/^3pc/[u]^{\stackanchor{special}{connected sum}}}
\]
\caption{A special decomposition and connected sum.} \label{fig_special_defn}
\end{figure}
\normalsize

The following important theorem of Matveev states that any nontrivial annular decomposition can be recovered on a minimal representative as either an annular or special decomposition. 

\begin{theorem}[Matveev \cite{matveev_roots}, Lemma 5] \label{thm_mat_kau_mant} If $\upsilon=\upsilon_1\# \upsilon_2$ is a nontrivial connected sum of virtual knots, then a minimal representative of $\upsilon$ has either an annular decomposition or a special decomposition into representatives of $\upsilon_1$ and $\upsilon_2$.
\end{theorem}

\section{Compressible annuli, minimality, and primeness} \label{sec_compress}

In this section, we prove some elementary facts about minimal representatives, locally trivial knots in thickened surfaces, and prime virtual knots. Together they provide a strategy that will be used to prove Theorems \ref{thm_A}, \ref{thm_B}, and \ref{thm_C}. These results are most easily phrased in terms of compressible vertical annuli in $\Sigma \times I$. First we will review compressible surfaces and vertical annuli in Section \ref{sec_comp_destab}. In Section \ref{sec_comp_decomp}, we investigate the vertical annuli that occur in decompositions and special decompositions. The strategy for proving primeness is given last in Section \ref{sec_prov_prime}.

\subsection{Compressible surfaces and vertical annuli} \label{sec_comp_destab} Let $M$ be a compact orientable $3$-manifold. An embedded circle $C$ in an embedded surface $S\subset M$ is said to be \emph{inessential} if there is a $2$-disc $D$ embedded in $S$ such that $\partial D=C$. An embedded circle that is not inessential will be called \emph{essential}. An embedded surface $S\subset M$ is \emph{compressible} if there is a $2$-disc $D$ in $M$ such that $D \cap S=\partial D$ is essential in $S$. Otherwise it is called \emph{incompressible}. If every boundary component of $M$ is incompressible, $M$ is said to be \emph{boundary irreducible} (or briefly, $\partial$-\emph{irreducible}). If every $2$-sphere embedded in $M$ bounds a $3$-ball in $M$, then $M$ is called \emph{irreducible}.

Now consider the case of the $3$-manifold $\Sigma \times I$, where $\Sigma$ is closed and oriented. First note that $\Sigma \times I$ is boundary irreducible. Furthermore, $\Sigma \times I$ is irreducible if and only if $\Sigma$ is not a $2$-sphere. Henceforth, any embedded surfaces we encounter in $\Sigma \times I$ will be assumed to be two-sided and orientable. If two surfaces intersect in $\Sigma \times I$, we will furthermore assume those intersections are transversal. A typical argument in this paper will analyze the intersections of two surfaces in $\Sigma \times I$ where at most one of the surfaces has non-empty boundary (e.g. an annulus intersecting a sphere or torus). Intersections will thus be assumed to consist of a disjoint union of embedded circles. 

A properly embedded annulus $A \subset \Sigma \times I$ is \emph{vertical} if $p^{-1}(p(A))=A$, where $p:\Sigma \times I \to \Sigma$ is projection onto the first factor. We will also say that any annulus isotopic to a vertical annulus is vertical. Any annulus used in a destabilization, decomposition, or special decomposition is vertical. We record the following fact, without proof, for future reference. 

\begin{lemma} \label{lemma_ball} If $A$ is a compressible vertical annulus in $\Sigma \times I$, then there are $2$-discs $D_0 \subset \Sigma \times \{0\}$ and $D_1 \subset \Sigma \times \{1\}$ such that $A \cup D_0 \cup D_1$ is a $2$-sphere bounding a $3$-ball $B$ in $\Sigma \times I$ with $B \cap \partial(\Sigma \times I)=D_0 \cup D_1$. 
\end{lemma}

The next technical lemma will be used frequently in the proofs of Theorems \ref{thm_A} and \ref{thm_B}.

\begin{lemma} \label{lemma_experiment_2} Let $A$ be an incompressible vertical annulus in $\Sigma \times I$, $\Sigma \ne S^2$. Suppose that $F$ is a closed orientable surface transverse to $A$ that bounds a compact orientable $3$-manifold $N$ embedded in the interior of $\Sigma \times I$. Furthermore, suppose that there is an embedded annulus $\alpha \subset A$ in the exterior of $N$ such that one component of $\partial \alpha$ is in $F$ and the other coincides with a component of $\partial A$. Then $C$ is essential in $F$ and does not bound a disc embedded in $N$.
\end{lemma}

\begin{proof} Suppose $C$ bounds a disc $D$ in either $F$ or $N$. Then the disc $\alpha \cup D$ is embedded $\Sigma \times I$ with $\partial (\alpha \cup D) \subset \partial (\Sigma \times I)$. Since $\Sigma \ne S^2$, $\Sigma \times I$ is boundary irreducible and hence $\partial (\alpha \cup D)$ bounds a disc in $\partial (\Sigma \times I)$. Since $A$ is vertical, this implies the contradiction that $A$ is compressible. 
\end{proof}

Finally in this section we have the following well-known result that characterizes minimal representatives of virtual knots in terms of compressible destabilizing annuli.

\begin{theorem} \label{thm_compress} $K \subset \Sigma \times I$ is a minimal representative of a virtual knot if and only if every destabilizing annulus of $K$ is compressible in $\Sigma \times I$.
\end{theorem}

\begin{proof} By Lemma \ref{lemma_ball}, if every destabilizing annulus is compressible, then $K$ admits no genus reducing destabilizations. Conversely, Theorem \ref{thm_kuperberg} implies that every destabilizing annulus of a minimal genus representative of $K$ must intersect $\Sigma \times \{1\}$ in an inessential curve. 
\end{proof}

\subsection{Compressible decompositions} \label{sec_comp_decomp} Compressible annular decompositions can be interpreted in terms of \emph{local knots}. Recall that for a knot $K$ in a $3$-manifold $M$, a \emph{splitting-$S^2$} is a $2$-sphere $S$ smoothly embedded in $\mathring{M}$ that intersects $K$ transversely in two points. If a splitting-$S^2$ bounds a $3$-ball $B^3 \subset M$, then $B \cap K$ is called a \emph{local knot}. A local knot is said to be \emph{trivial} if $K$ intersects $B$ in an unknotted arc. If $M$ is irreducible, we will say that $K$ is \emph{locally trivial} if every splitting-$S^2$ of $K$ bounds a $3$-ball $B$ that intersects $K$ in a trivial local knot (see e.g. \cite{kawauchi}, Definition 3.2.2).

By Lemma \ref{lemma_ball}, every compressible annular decomposition corresponds to a local knot. Conversely, suppose $S$ is a splitting-$S^2$ bounding a $3$-ball $B$ in $\Sigma \times I$. Contract $B$ to a sufficiently small ball so that $B \subset D \times I$ for some disc $D\subset \Sigma$. By an isotopy, it may be arranged so that $(\partial D) \times I$ intersects the knot $K$ transversely in exactly two points. The vertical annulus $\partial D \times I$ is compressible and defines an annular decomposition of $K$. The following theorem states that representatives of non-classical prime virtual knots are locally trivial. The converse is false; there are locally trivial knots in $\Sigma \times I$ that represent composite virtual knots (see e.g. Example \ref{example_hyp_not_prime} ahead).

\begin{theorem} \label{thm_prime_implies_loc_triv} A representative $K \subset \Sigma \times I$ of non-classical prime virtual knot $\upsilon$ is locally trivial.
\end{theorem}
\begin{proof} Since $K$ is non-classical, $\Sigma \ne S^2$ and $\Sigma \times I$ is irreducible. Suppose that $K$ is not locally trivial. Then there is a splitting-$S^2$ of $K$ that bounds a $3$-ball $B$ in $\Sigma \times I$ such that $K \cap B$ is a nontrivial arc. By the above remarks, there is a decomposition $\upsilon=\upsilon_1 \#\upsilon_2$, where $\upsilon_1$ is a nontrivial classical knot corresponding to $K \cap B$. Since $\upsilon$ is prime, $\upsilon \leftrightharpoons \upsilon_1$. As this contradicts the fact that $\upsilon$ is non-classical, $K$ is locally trivial. \end{proof}

The following technical lemma will be used in the course of the proof of Theorem \ref{thm_A}. It gives a sufficient condition for a decomposing annulus of a minimal representative to be compressible.

\begin{lemma} \label{lemma_tech_compress} Let $K \subset \Sigma \times I$ be a minimal representative of a non-classical virtual knot $\upsilon$ and let $A$ be an annular decomposition for $K$. Suppose that there is a splitting-$S^2$, $S$, of $K$ and $2$-discs $D,D' \subset \Sigma \times I$ satisfying the following properties:
\begin{enumerate}
\item $D \cup D'=S$ and $D \cap D'=\partial D=\partial D' \subset A$,
\item $D \subset A$, and
\item $D\cap K=A \cap K$.
\end{enumerate}
Then $A$ is compressible in $\Sigma \times I$.
\end{lemma}
\begin{proof} Since $K$ is non-classical, $\Sigma \ne S^2$ and $\Sigma \times I$ is irreducible. A splitting-$S^2$ of $K$ has only two intersections with $K$. Since $|D\cap K|=|A\cap K|=2$, $D' \cap K=\emptyset$. Suppose $S$ is chosen to minimize the number of components of $D' \cap A$. If $D'\cap A=\partial D'$, let $A'=(A \smallsetminus D) \cup D'$. Then $A'$ is a destabilization annulus of $K$ and $A'$ is compressible by Theorem \ref{thm_compress}. Since $\Sigma \times I$ is irreducible, the $2$-sphere $D \cup D'$ bounds a $3$-ball and hence there is an isotopy from $A$ to $A'$ that pushes $D$ to $D'$. Thus, $A$ is also compressible.

Suppose that $C \subset D' \cap A$ is an innermost inessential circle in $A \smallsetminus \mathring{D}$. Since $C \cap D=\emptyset$, $C$ bounds a disc $D'' \subset A$ that is disjoint from $D$. On $D'$, $C$ bounds a disc $D'''$ such that $D'' \cap D'''=C$. Again invoking the the fact that $\Sigma\times I$ is irreducible, we have that $D'' \cup D'''$ bounds a $3$-ball in $\Sigma \times I$. The number of components of $D' \cap A$ can be reduced by pushing $D'''$ to $D''$. This contradicts the choice of $S$. Therefore, assume there are no components $C \subset D' \cap A$ that are inessential in $A\smallsetminus \mathring{D}$.

Now, consider a circle $C \subset D' \cap A$ that is innermost on $D'$, $C \ne \partial D'$. Then either $C$ is parallel to $\partial D$ in $A$ or $C$ is essential in $A$. Let $D''$ be the disc in $D'$ bounded by $C$. If $C$ is essential in $A$, then $D''$ is a compression disc for $A$ and the result follows. If $C$ is parallel to $\partial D$, let $D'''\subset A$ be the $2$-disc bounded by $C$, so that $D \subset D'''$. Then $A'=(A \smallsetminus D''') \cup D''$ is a destabilizing annulus for $K$.  By Theorem \ref{thm_compress}, $A'$ is compressible and it follows as above that $A$ is also compressible.    
\end{proof}

Lastly, we consider the vertical annuli in special decompositions, which are never compressible.

\begin{lemma} \label{lemma_special_incompress} If $A=A_1 \sqcup A_2$ is a pair of vertical annuli in a special decomposition, then both $A_1$ and $A_2$ are incompressible in $\Sigma \times I$.
\end{lemma}
\begin{proof} If $A_i$ is compressible, then Lemma \ref{lemma_ball} implies that $A_i$ is separating. But this is impossible as $A_i$ intersects $K$ only once.
\end{proof}

\subsection{Proving primeness} \label{sec_prov_prime} By definition, a knot in the $3$-sphere $S^3$ is prime if and only if it is locally trivial in $S^3$. As discussed above, local triviality of a representative of a virtual knot $\upsilon$ is insufficient to guarantee that $\upsilon$ is prime. The following theorem adds conditions to correct for this.

\begin{theorem} \label{thm_prime} Suppose $K\subset \Sigma \times I$ is a minimal representative of nontrivial virtual knot $\upsilon$. If $K$ is locally trivial, admits only compressible annular decompositions, and admits no special decompositions, then $\upsilon$ is a prime virtual knot. 
\end{theorem}

\begin{remark} The converse of Theorem \ref{thm_prime} is false. For example, the virtual trefoil in Figure \ref{fig_carter} is prime by Akimova-Matveev \cite{akimova_matveev} but admits special decompositions.
\end{remark}

\begin{proof} Suppose that $\upsilon$ is not prime, so that there is a non-trivial decomposition $\upsilon \leftrightharpoons \upsilon_1 \# \upsilon_2$. By Theorem \ref{thm_mat_kau_mant}, $K \subset \Sigma \times I$ inherits either an annular or special decomposition. By hypothesis, $K$ admits no special decompositions. Suppose the annulus $A$ is a decomposition of $K$, $K=K_1 \# K_2 \subset (\Sigma_1 \#\Sigma_2) \times I$, where $K_i$ represents $\upsilon_i$ for $i=1,2$.  By hypothesis, $A$ is compressible in $\Sigma \times I$. The annular decomposition gives a local knot in $\Sigma \times I$. Without loss of generality, assume that the local knot corresponds to $K_2$, so that $\Sigma_2=S^2$ and $\Sigma_1=\Sigma$. Since $K$ is locally trivial, the local knot is an unknotted arc. Then $K_2$ is itself trivial. Moreover, $K$ is ambient isotopic to $K_1$. Thus the decomposition $\upsilon \leftrightharpoons \upsilon_1 \# \upsilon_2$ is trivial, which is a contradiction.
\end{proof}

Since $S^2 \times I$ is not irreducible, the case of classical knots must be treated separately from the case of non-classical virtual knots.  

\begin{theorem} \label{thm_class_virt_prime} Let $\upsilon$ be a nontrivial virtual knot and suppose $\upsilon$ is equivalent to a classical knot $K \subset S^3$. Then $\upsilon$ is a prime virtual knot if and only if $K$ is locally trivial in $S^3$. 
\end{theorem}

\begin{proof} Since $\upsilon$ is classical and nontrivial, it has a representative $K \subset S^2 \times I$, where $K$ is not the unknot. Attach $3$-balls to $\partial (S^2 \times I)$ along their boundaries to obtain a copy of $S^3$. All vertical annuli in $S^2 \times I$ are compressible, so $K$ admits no special decompositions by Lemma \ref{lemma_special_incompress}. If $A$ is an annular decomposition of $K$, then Lemma \ref{lemma_ball} implies that $A$ separates $S^2 \times I$ into two $3$-balls. If $K$ is locally trivial in $S^3$, then one of these $3$-balls must intersect $K$ in an unknotted arc. Hence, every annular decomposition is trivial and $\upsilon$ is a prime virtual knot. Conversely, a splitting-$S^2$ of $K$ in $S^3$ can be deformed by an innermost circle argument so that it intersects $S^2 \times I$ in a vertical annulus $A$. Since $\upsilon$ is a prime virtual knot, the annular connected sum defined by $A$ must be trivial. Therefore, one of the the two local knots defined by the compressible annulus $A$ is trivial.
\end{proof}

\section{Prime Satellite Virtual Knots} \label{sec_sat}

In this section, we prove Theorem \ref{thm_A}, that every virtual knot is concordant to a prime satellite virtual knot. In Section \ref{sec_sat_meridians}, we give a review of meridians, longitudes, and linking numbers of knots in $\Sigma \times I$.  Section \ref{sec_sat_prime} gives necessary and sufficient conditions for a satellite knot in $\Sigma \times I$ to represent a prime virtual knot. The proof of Theorem \ref{thm_A} is given last, in Section \ref{sec_proof_1}.

\subsection{Meridians and longitudes} \label{sec_sat_meridians} For an oriented knot $J \subset \Sigma \times I$, let $N(J)$ denote a regular neighborhood of $J$ in the interior of $\Sigma \times I$. By a \emph{meridian} of $J$, we mean an essential embedded circle on $\partial N(J)$ that bounds a disc in $N(J)$. Note that $H_1(\Sigma \times I \smallsetminus J,\Sigma \times \{1\})\cong \mathbb{Z}$ and is generated by an oriented meridian $\mu$ of $J$ (see \cite{acpaper}, Proposition 7.1). For an oriented knot $K \subset \Sigma \times I$ disjoint from $J$, the \emph{linking number of} $J$ \emph{and} $K$ \emph{in} $\Sigma \times I$ is the integer $m$, where $[K]=m \cdot \mu \in H_1(\Sigma \times I \smallsetminus J,\Sigma \times \{1\})$. The linking number is denoted by $\lk_{\Sigma}(J,K)$. The linking number can be computed from a diagram of $J \sqcup K$ on $\Sigma$ by counting, with sign, the number of times that the first component $J$ crosses over the second component $K$. In particular, if $K$ has only over crossings with $J$, $\lk_{\Sigma}(J,K)=0$.  

A \emph{longitude} of a solid torus $N$ is an essential embedded circle on $\partial N$ that generates $H_1(N;\mathbb{Z})$. If $N=N(K)$ is a regular neighborhood of a knot $K \subset \Sigma \times I$, then a longitude of $N(K)$ cobounds an annulus in $N$ with $K$. A \emph{longitude $\ell$ of the knot} $K \subset \Sigma \times I$ is an essential embedded circle  in $\partial N(K)$, such that $[\ell]=0$ in $H_1(\overline{\Sigma \times I\smallsetminus N(K)}, \Sigma \times \{1\};\mathbb{Z})$. The following lemma, whose proof we omit as it is the same as in the classical case, relates a longitude of a regular neighborhood of $K$ to a longitude of the knot $K$. An illustration is given in Figure \ref{fig_long}.

\begin{lemma}\label{lemma_long} Let $K \subset \Sigma \times I$ be a knot and $\ell \subset \partial N(K)$ be an embedded circle. Then $\ell$ is a longitude of the knot $K$ if and only if $\lk_{\Sigma}(K,\ell)=0$ and $\ell$ is a longitude of $N(K)$. Moreover, if $\ell$ is a longitude of the knot $K$, then $\ell$ and $K$ cobound an annulus in $N(K)$.
\end{lemma}

\begin{figure}[htb]
\begin{tabular}{c} \def\svgwidth{1.75in}
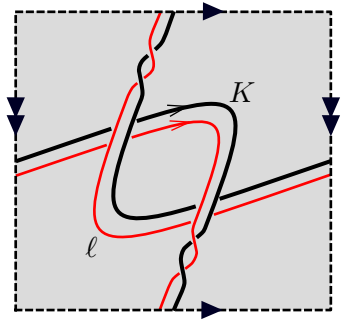 \end{tabular}
\caption{A longitude $\ell$ of a knot $K$ in the thickened torus.} \label{fig_long}
\end{figure} 

This subsection concludes with two lemmas that will be used in the proof of Theorem \ref{thm_A}.

\begin{lemma}\label{lemma_long_triv} Suppose that $\ell$ is a longitude of $K \subset \Sigma \times I$ and that $\ell$ bounds an embedded disc in $\overline{\Sigma \times I \smallsetminus N(K)}$. Then $K$ represents a trivial virtual knot.
\end{lemma}
\begin{proof} By Lemma \ref{lemma_long}, there is an annulus $A$ in $N(K)$ such that $\partial A=K \sqcup \ell$. Attach the disc that $\ell$ bounds in $\Sigma \times I \smallsetminus N(K)$ to this annulus. This also results in a disc. Since $K$ bounds a disc in $\Sigma \times I$, it represents the unknot.
\end{proof}

\begin{lemma} \label{lemma_unknot} If a knot $K \subset \Sigma \times I$ cobounds an embedded annulus $\alpha \subset \Sigma \times I$  with a smooth simple closed curve $\gamma$ in $\partial (\Sigma \times I)$, then $K$ is a representative in $\Sigma \times I$ of the trivial virtual knot.
\end{lemma}

\begin{proof} For some $j=0,1$, $\gamma \subset \Sigma \times \{j\}$. Let $A$ be a closed tubular neighborhood of $\gamma$ in $\Sigma \times \{j\}$ and let $U=A \times I$. Let $H:S^1 \times I \to \Sigma \times I$ be a parametrization of the embedding of $\alpha$ into $\Sigma \times I$. Then there is an $\varepsilon>0$ such that $H(S^1 \times [1-\varepsilon,1]) \subset U$. View the map $H$ as an isotopy taking $K$ to the knot $K'=H(S^1 \times \{1-\varepsilon\}) \subset U$. By the isotopy extension theorem, $K$ is ambient isotopic to $K'$. Now, $\partial A \times I$ consists of two vertical annuli each of which is disjoint from $K'$. Destabilizing along these annuli gives $K'$ as a knot in $S^2 \times I$. The curve $\gamma$, now contained in $S^2 \times \{j\}$, bounds a disc $D$ in $S^2 \times \{j\}$. Attach $D$ to the annulus $H(S^1 \times [1-\varepsilon,1])$ (or more exactly, its image after destabilization). This gives a disc $D'$ with $\partial D'=K'$. Pushing $D'$ into the interior of $S^2 \times I$ shows that $K'$ bounds a disc in $S^2 \times I$ and hence $K'$ is trivial.  
\end{proof}

\subsection{Prime satellite virtual knots} \label{sec_sat_prime} Let $P$ be a knot in $S^1 \times B^2$ and $K \subset \Sigma \times I$, where $P$ is not contained in a $3$-ball in the interior of $S^1\times B^2$. Let $\psi:S^1 \times B^2 \to N(K)$ be a diffeomorphism mapping an oriented meridian of $S^1 \times B^2$ to an oriented meridian of $N(K)$ and the longitude $S^1\times \{1\}$ of $S^1 \times B^2$ to a longitude of $N(K)$. Then $\psi(P)$ is called a \emph{satellite knot with pattern $P$ and companion $K \subset \Sigma \times I$}. If a virtual knot has a representative in some thickened surface $\Sigma \times I$ that is a satellite knot in $\Sigma \times I$, we will say that it is a \emph{virtual satellite knot}. The following theorem relates the virtual genus of a virtual satellite knot to that of its companion.

\begin{figure}[htb]
\begin{tabular}{ccc} \begin{tabular}{c} \def\svgwidth{1.75in}
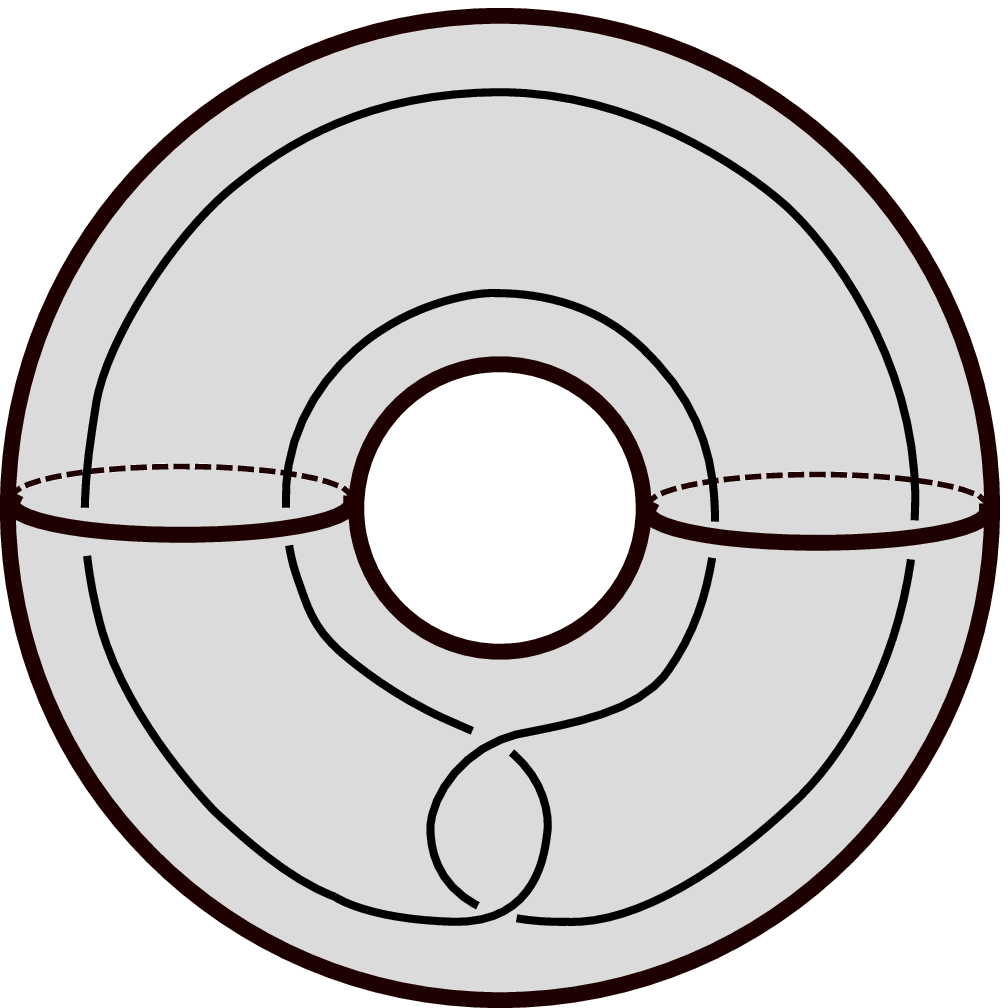 \end{tabular} & \begin{tabular}{c} \def\svgwidth{1.75in}
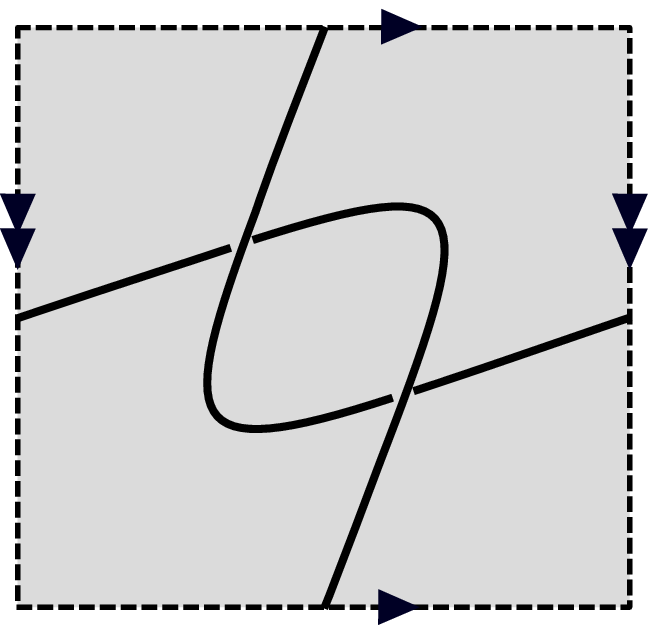 \end{tabular} & \begin{tabular}{c} \def\svgwidth{1.75in}
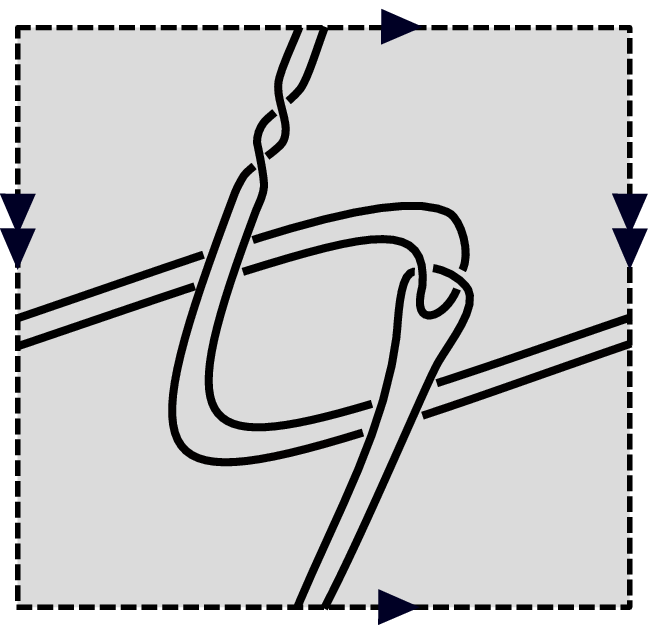 \end{tabular} \end{tabular}
\caption{A satellite knot (right) with companion $K$ (center) and pattern $P$ (left).} \label{fig_sat_defn}
\end{figure} 

\begin{theorem}[Silver-Williams \cite{sil_will_sat}, Theorem 1.2] \label{thm_sw} A satellite knot with companion $K \subset \Sigma\times I$ has the same virtual genus as $K$.
\end{theorem}

\begin{example} \label{ex_connect_not_sat} Every connected sum $K_1\# K_2$ of classical knots can be viewed as a satellite with pattern $K_2$ and companion $K_1$ (or vice versa). A connected sum $\upsilon_1 \#\upsilon_2$ of virtual knots need not be a satellite with companion equivalent to either $\upsilon_1$ or $\upsilon_2$. Indeed, a connected sum of virtual knots satisfies $g(\upsilon_1 \# \upsilon_2) \ge g(\upsilon_1)+g(\upsilon_2)-1$, where $g(\upsilon)$ denotes the virtual genus of $\upsilon$ (see \cite{kauffman_manturov}, Theorem 4). Thus, as long as $g(\upsilon_1),g(\upsilon_2) \ge 2$, Theorem \ref{thm_sw} implies that a connected sum cannot be a satellite with companion equivalent to either $\upsilon_1$ or $\upsilon_2$.  
\end{example}

For a knot $P \subset S^1 \times B^2$, the \emph{winding number} is the absolute value of the integer $[P] \in H_1(S^1 \times B^2) \cong \mathbb{Z}$. The \emph{wrapping number} is the minimum number of intersections of $P$ with some disc $D \subset S^1 \times B^2$, where the minimum is taken over all discs $D$ such that $\partial D$ is a meridian of $S^1 \times B^2$. The next result is the main theorem of this section. It states that when a pattern has sufficiently large wrapping number, the satellite of a non-classical virtual knot is prime if and only if the pattern is locally trivial. This extends Livingston's theorem (\cite{livingston}, Theorem 4.2) to the non-classical case.

\begin{theorem}\label{lemma_conc_to_sat} Let $P \subset S^1 \times B^2$ have wrapping number greater than one and let $K \subset \Sigma \times I$ be a minimal representative of a non-classical virtual knot $\upsilon$. Then a satellite knot with pattern $P$ and companion $K \subset \Sigma \times I$ represents a prime virtual knot if and only if $P$ is locally trivial in $S^1 \times B^2$.
\end{theorem}

\begin{proof} If $P$ is not locally trivial, then there is a $3$-ball $B$ in the interior of $S^1 \times B^2$ such that $S=\partial B$ is a splitting-$S^2$ and $B \cap P$ is nontrivial. In other words, $P$ contains a nontrivial local knot. Then the satellite with pattern $P$ and companion $K$ must also contain a nontrivial local knot and hence cannot be prime. The converse follows from  Theorem \ref{thm_prime} and the next lemma.
\end{proof}

\begin{lemma} \label{lemma_lemma} Let $K \subset \Sigma \times I$ be a minimal representative of a non-classical virtual knot $\upsilon$. Let $P \subset S^1 \times B^2$ be a locally trivial knot with wrapping number greater than one and let $K' \subset \Sigma \times I$ be a satellite knot with pattern $P$ and companion $K$.  Then:
\begin{enumerate}
\item $K'$ is locally trivial,
\item $K'$ admits only compressible annular decompositions, and
\item $K'$ admits no special decompositions.
\end{enumerate}
\end{lemma}

\begin{remark} Some paragraphs below are marked with stars ($\star \cdots \star$) for future reference. \end{remark}

\begin{proof}[Proof of Lemma \ref{lemma_lemma} (1)]  Let $N=N(K)$ be a regular neighborhood of $K$ defining the satellite. Let $S$ be a splitting-$S^2$ of $K' \subset \Sigma \times I$. Let $A$ be the annulus obtained by deleting small neighborhoods of $S \cap K'$ from $S$. Suppose $S$ is chosen so that $A \cap \partial N$ has the minimal number of connected components among all splitting-2-spheres $S'$ such that $(S',S'\cap K')$ is isotopic as a pair to $(S,S\cap K')$, with $K'$ fixed set-wise. We will show $S \subset \text{int}(N)$, so that local triviality of $P$ implies local triviality of $K'$.

($\star$) If $A \cap \partial N$ has a component that is inessential in $A$, choose $C$ to be an innermost one and let $D$ be a disc in $A$ with $C=\partial D$. Then either $D \subset \overline{\Sigma \times I \smallsetminus N}$ or $D \subset N$. We argue that in either case, $C$ must bound a disc in $\partial N$. Suppose by way of contradiction that $C$ is essential in $\partial N$, so that $[C]\ne 0$ in $H_1(\partial N)$. If $D \subset \overline{\Sigma \times I \smallsetminus N}$, then it follows that $\lk_{\Sigma}(K,C)=0$ and $C$ is a longitude of $K$. Lemma \ref{lemma_long_triv} then implies the contradiction that $\upsilon$ is the trivial knot. If $D \subset N$, then $C$ is a meridian of $N$. $D$ must then intersect $P$ and hence $K'$. This contradicts the fact that $D \subset A$ has no intersections with $K'$. Thus it follows in either case that $C$ bounds a disc $D' \subset \partial N$. Since $\upsilon$ is non-classical, $\Sigma \ne S^2$ and $\Sigma \times I$ is irreducible. The $2$-sphere $D \cup D'$ then bounds a $3$-ball $B \subset \Sigma \times I$. Since $K'$ is non-classical (by Theorem \ref{thm_sw}), $K'$ is not contained in $B$ and thus $K' \cap B=\emptyset$. Thus the component $C$ may be removed by an isotopy of $D$ through $B$ that fixes $K'$. As this contradicts the minimality of $S$, it follows that $A \cap \partial N$ contains no components that are inessential in $A$.

Thus, all components of $A \cap \partial N$ are $\partial$-parallel in $A$. Suppose $C \subset A \subset S$ is an innermost one, i.e. bounding a disc $D$ in $S$ containing one point only of $K' \cap S$. Since $\partial N$ is separating, $D \subset N$. Moreover, $C$ is a meridian of $N$. Indeed, if $C$ is inessential on $\partial N$, there would be a disc $D'\subset \partial N$ such that $D \cup D'$ is a $2$-sphere in $\Sigma \times I$ intersecting $K'$ in only one point. This is impossible since $\Sigma \ne S^2$ and $\Sigma \times I$ is irreducible. Then since $C$ is a meridian and the wrapping number of $P$ is at least 2, $D \cap K'$ cannot contain only one point. It follows that $A \cap \partial N=\emptyset$. Therefore, $S \subset \text{int}(N)$. Since $P$ is locally trivial, $S$ bounds a $3$-ball containing an unknotted arc of $K'$.
\end{proof}

\begin{figure}[htb]
\begin{tabular}{ccc} & & \\ & & \\ \begin{tabular}{c} \def\svgwidth{2.5in}
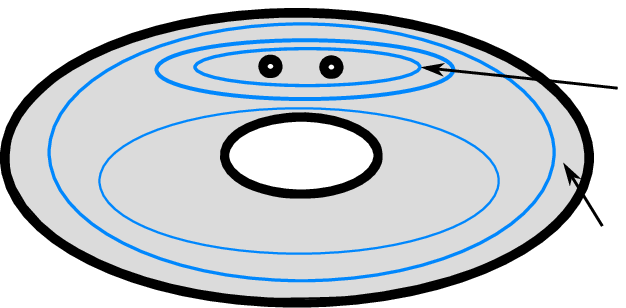 \end{tabular} & & \begin{tabular}{c} \def\svgwidth{2.5in}
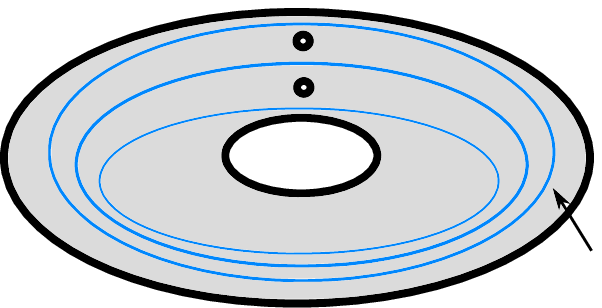 \end{tabular} \end{tabular}
\caption{The space $Y$ is the decomposing annulus $A$ with the intersections with $K'$ removed. The blue circles are schematic depictions of the components $A \cap \partial N$.} \label{fig_lemma_lemma_2}
\end{figure} 

\begin{proof}[Proof of Lemma \ref{lemma_lemma} (2)] By Theorem \ref{thm_sw}, the satellite knot $K'$ is a minimal representative. Suppose that there is an incompressible decomposing annulus $A \subset \Sigma \times I$ of $K'$. Let $Y$ be the annulus with two holes obtained by deleting from $A$ small neighborhoods of $A \cap K'$. Let $C \subset Y \cap \partial N$ be a component circle. As in the proof of Lemma \ref{lemma_lemma} (1), assume that $A$ is chosen to minimize the number of connected components of $A \cap \partial N$. 

If $C$ is inessential in $Y$, then it follows as in paragraph ($\star$) above that an innermost such $C$ bounds a disc on $\partial N$. As before, the existence of such $C$ contradicts either the minimality of $A$ or the fact that $\upsilon$ is non-classical. Assume then that there are no curves $C$ inessential in $Y$.

($\star\star$) We claim that there are also no components $C \subset Y \cap \partial N$ that are $\partial$-parallel in $Y$ but not in $A$. Such a $C$ bounds a disc $D$ in $A$ intersecting $K'$ in one point. Suppose that $C$ is innermost in $A$. We will show that this $C$ cannot exist by proving that it is a meridian of $N$. Since the wrapping number of $P$ is at least 2, $D$ must then intersect at least two points. To see this, first note that $D \subset N$. If $C$ bounds a disc $D' \subset \partial N$, then $D \cup D'$ is a $2$-sphere bounding a $3$-ball where $D\cup D'$ intersects $K'$ only once. As this contradicts the irreducibility of $\Sigma \times I$, $C$ must be essential in $\partial N$. Hence, $C$ is a meridian of $N$ and the claim is proved.

It follows that either all components of $A \cap \partial N$ are $\partial$-parallel in $A$ (as in Figure \ref{fig_lemma_lemma_2}, right) or there exists a component of $A \cap \partial N$ that is $\partial$-parallel in neither $A$ nor $Y$ (as in Figure \ref{fig_lemma_lemma_2}, left). 

($\star\star\star$) Next we show that in both cases in Figure \ref{fig_lemma_lemma_2}, there can be no components $C \subset A \cap \partial N$ that are $\partial$-parallel in $A$. Suppose that there is such a $C$. Choose $C$ to be outermost, so that $C$ and a component of $\partial A$ cobound an annulus $\alpha\subset A$ that contains no other components of $Y \cap \partial N$. Note that an outermost component must exist in both cases of Figure \ref{fig_lemma_lemma_2}. Furthermore, note that the annulus $\alpha$ for an outermost $C$ must satisfy $\alpha \subset \overline{\Sigma \times I \smallsetminus N}$ and $\alpha \cap K'=\emptyset$. This can be seen by observing that $\partial N$ is separating, $K' \subset N$, and $\partial A \subset \overline{\Sigma \times I \smallsetminus N}$. Now, the annulus $\alpha$ can be used to construct an ambient isotopy taking $C$ to a knot in $\overline{\Sigma \times I \smallsetminus N}$ that lies above $K$ in $\Sigma \times I$. It follows that $\lk_{\Sigma}(K,C)=0$. Since $A$ is incompressible and $K$ is minimal, Lemma \ref{lemma_experiment_2} implies that $C$ is essential on $\partial N$. Lemma \ref{lemma_long} then implies that $C$ is a longitude of $K$. Gluing the annulus that $K$ and $C$ cobound in $N$ to $\alpha \subset \overline{\Sigma \times I \smallsetminus N}$, we see that $K$ cobounds an annulus with a simple closed curve in $\partial (\Sigma \times I)$. By Lemma \ref{lemma_unknot}, $K$ is a representative in $\Sigma \times I$ of the unknot. As this contradicts the hypothesis on $K$, there are no components $C \subset A \cap \partial N$ that are $\partial$-parallel in $A$.

Thus, if there is an incompressible decomposing annulus $A$ of $K'$, all components $C \subset A \cap \partial N$ must be $\partial$-parallel in neither $A$ nor $Y$. Such curves form a set of concentric inessential circles $\{C_0, C_1,\ldots,C_n\}$ in $A$, with $C_0$ innermost in $A$. We will show that this set is in fact empty. First we claim that if the set is nonempty, then all $C_i$ are meridians of $N$. If $C_i$ is inessential on $\partial N$, assume that $C_i$ is innermost in $\partial N$. Then $C_i$ bounds a disc $D\subset \partial N$ that contains no $C_j$ for $j \ne i$. Each $C_i$ bounds a disc $D_i \subset A$ intersecting $K'$ in two points. It follows that $D \cup D_i$ is a splitting-$S^2$ of $K'$. Since $K'$ is a minimal representative, Lemma \ref{lemma_tech_compress} implies that $A$ is compressible, which is a contradiction. It may therefore be assumed that $C_0,\ldots,C_n$ are essential on $\partial N$. To prove that $C_i$ is a meridian, it need only be shown $C_i$ bounds a disc in $N$. Certainly $D_0 \subset N$,  since $\partial N$ is separating, $K' \cap A \subset D_0$, and $K' \subset N$. Next recall that $C_i \cap C_0=\emptyset$ for all $i \ne 0$ and each $C_i$ is essential in $\partial N$. Since $C_0$ is a meridian of $N$, this implies that each $C_i$ is parallel to $C_0$ in $\partial N$. Therefore, $\{C_0,\ldots,C_n\}$ contains only meridians of $N$.

Finally, we show that $\{C_0,\ldots,C_n\}$ is empty. Since $\partial N$ is separating, there must be at least one component of $A \smallsetminus \partial N$ in $\text{int}(N)$ and one component the complement of $N$. It follows that the set $\{C_0,\ldots,C_n\}$ of concentric circles on $A$ has an odd number of elements. As each $C_i$ is a meridian of $N$, they divide $\partial N$ into an odd number of annuli. Enumerate the distinct annuli in $\partial N$ consecutively by $\nu_1,\ldots,\nu_{n+1}$, so that $\partial \nu_i=C_{i-1}\cup C_i$ when the indices are taken modulo $n+1$. Since $A$ is separates $\Sigma \times I$ into two components, it follows that $n \ge 1$. Furthermore, the interiors $\nu_i$ and $\nu_{i+1}$ cannot be in the same component of $\Sigma \times I \smallsetminus A$. As the number of annuli $\nu_i$ is odd, it must be that the interiors of $\nu_1$ and $\nu_{n+1}$ lie in the same component of $\Sigma \times I \smallsetminus A$. However, $\nu_1 \cap \nu_{n+1}=C_0$ and hence their interiors must lie in different components of $\Sigma \times I \smallsetminus A$. This contradiction implies $\{C_0,\ldots,C_n\}$ is empty. 

Thus, $A \cap \partial N=\emptyset$ for any incompressible decomposition annulus of $K'$. Since $\partial N$ is separating, $K' \subset N$ and $K'\cap A \ne \emptyset$, this is impossible. Thus, $K'$ admits no incompressible annular decompositions and the proof of Lemma \ref{lemma_lemma} (2) is complete.
\end{proof}

\begin{proof}[Proof of Lemma \ref{lemma_lemma} (3)] Suppose that $T=A_1 \sqcup A_2$ defines a special decomposition, where $A_1,A_2$ are vertical annuli intersecting $K'$ once each. By Lemma \ref{lemma_special_incompress}, $A_1$ and $A_2$ are incompressible in $\Sigma \times I$. Let $Y,Y_1,Y_2$ be $T,A_1,A_2$, respectively, with a small neighborhood of $T\cap K'$ removed. Then $Y=Y_1\sqcup Y_2$. Assume that $T$ has been chosen as in the proofs of Lemmas \ref{lemma_lemma} (1) and (2) so as to minimize the number of connected components of $T \cap \partial N$. 

Since $A_1,A_2$ are incompressible in $\Sigma \times I$, the proof mimics that of Lemma \ref{lemma_lemma} (2). The only difference is that $A_1,A_2$ each intersect $K'$ but once. Components $C\subset T \cap \partial N$ that are inessential in $Y$ may be removed by an isotopy, as described in paragraph ($\star$). By an argument identical to that of paragraph $(\star\star)$, there can be no components $C \subset Y \cap \partial N$ that are $\partial$-parallel in $Y$ but not in $T$. By an argument identical to that of paragraph $(\star\star\star)$, there can be no components $C \subset Y \cap \partial N$ that are $\partial$-parallel in $T$. Thus, $A\cap \partial N=\emptyset$. Since $\partial N$ is separating and $A \cap N \ne \emptyset$, this is impossible. Hence, $K'$ admits no special decompositions.  \end{proof}




\subsection{Proof of Theorem \ref{thm_A}} \label{sec_proof_1} First note that the unknot is a slice satellite of itself, so we may assume that $\upsilon$ is non-trivial. Recall the argument for a non-trivial classical knot $K \subset S^3$ given by Livingston \cite{livingston}. Let $P \subset S^1 \times B^2$ denote the knot depicted on the far left in Figure \ref{fig_sat conc_movie} and let $K'$ denote a satellite with pattern $P$. The movie in Figure \ref{fig_sat conc_movie} shows a concordance in $S^1 \times B^2$ to the knot $S^1 \times 0$ ($0 \in B^2 \subset \mathbb{C}$). This proves that $K'$ and $K$ are concordant. The knot $P$ as depicted can be considered as a knot in an unknotted solid torus in $S^3$. In $S^3$, $P$ is unknotted and by Livingston \cite{livingston}, Proposition 5.1, it follows that $P$ is nontrivial in $S^1 \times B^2$, locally trivial in $S^1 \times B^2$, and has wrapping number greater than $1$. Then Theorem 4.2 of \cite{livingston}, which is the classical version of our Lemma \ref{lemma_conc_to_sat}, proves that $K'$ is locally trivial (i.e. prime) in $S^3$.

Now suppose that $\upsilon$ is a virtual knot and let $K \subset \Sigma \times I$ be a minimal representative. Let $K'$ be a satellite of $K$ with pattern $P$. Then $K$ and $K'$ are concordant exactly as in the case of knots in $S^3$. If $\upsilon$ is classical, $K'\subset S^2 \times I$. By the above remarks, $K'$ is locally trivial in $S^3$ and Theorem \ref{thm_class_virt_prime} implies that $K'$ represents a prime virtual knot. If $\upsilon$ is a non-classical virtual knot, Theorem \ref{lemma_conc_to_sat} implies $K'$ is a prime virtual knot. By Theorem \ref{thm_sw}, $\upsilon'$ has the same virtual genus as $\upsilon$. Therefore, every non-classical virtual knot is concordant to a prime satellite virtual knot having the same virtual genus. \hfill $\square$

\begin{figure}[htb]
\begin{tabular}{c} \def\svgwidth{5.5in}
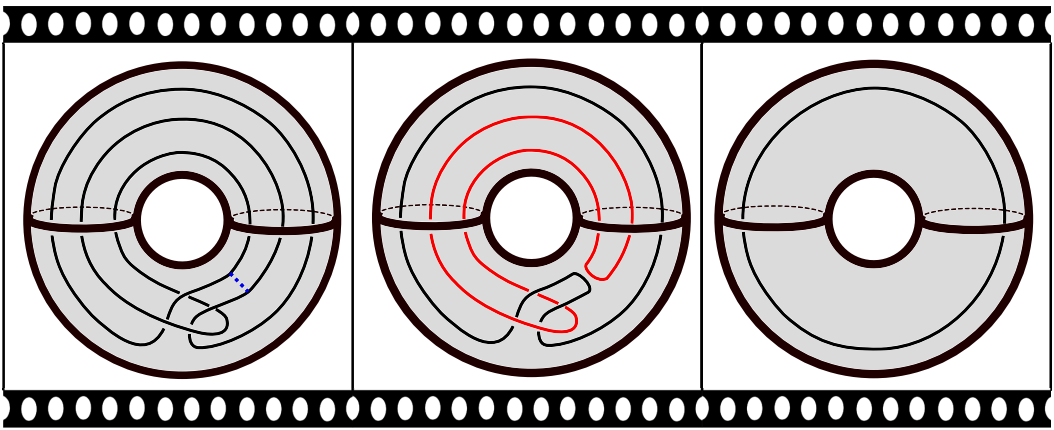 \end{tabular}
\caption{Concordance used in the proof of Theorem \ref{thm_A}.} \label{fig_sat conc_movie}
\end{figure} 

\section{Tangles, Complementary Tangles, and Hyberbolic Knots in $\Sigma \times I$} \label{sec_hyp}

The aim of this section is to prove Theorem \ref{thm_B}, that every virtual knot is concordant to a hyperbolic prime virtual knot. The main tool is a generalization of tangles in $3$-balls to \emph{complementary tangles} in $\Sigma \times I$. Hyperbolic knots in $\Sigma \times I$ will be constructed by joining together tangles in $B^3$ and complementary tangles in $\Sigma \times I$. Section \ref{sec_hyp_defn} reviews hyperbolic knots in $3$-manifolds and defines hyperbolic virtual knots. Section \ref{sec_tangles} reviews tangles in $B^3$. In Section \ref{sec_complementary}, we define complementary tangles and study their elementary properties. Section \ref{sec_tangle_ops} investigates joins of tangles and complementary tangles. The proof of Theorem \ref{thm_B} appears in Section \ref{sec_proof_2}.

\subsection{Hyperbolic knots and virtual knots} \label{sec_hyp_defn} Recall that a $3$-manifold is said to be \emph{Haken} if it is compact, orientable, irreducible, boundary irreducible, and sufficiently large (see e.g. Matveev \cite{matveev_book}). A $3$-manifold $M$ is said to be \emph{anannular} if every properly embedded incompressible annulus is $\partial$-parallel and \emph{atoriodal} if every incompressible torus in $M$ is $\partial$-parallel. A $3$-manifold is said to be \emph{simple} if it is compact, orientable, irreducible, boundary irreducible, atoroidal, and anannular. By Thurston's Hyperbolization Theorem \cite{thurston} and the Torus Theorem (Feustel \cite{feustel}, Theorem 4), a $3$-manifold that is both simple and Haken has hyperbolic interior.

A knot $K$ in a compact orientable $3$-manifold $M$ is said to be \emph{hyperbolic} if the interior of its exterior in $M$ is a hyperbolic $3$-manifold. If the exterior of $K$ is simple and Haken, then the volume will be finite if and only if $\partial M$ is a disjoint union of tori \cite{thurston}. If $\partial M$ contains no $2$-spheres and $M'=M\smallsetminus \{\text{torus boundary components}\}$ admits a hyperbolic metric that is totally geodesic on $\partial M'$, then a finite volume hyperbolic structure can be constructed by ``doubling'' $M'$ along $\partial M'$ (see e.g. Marden \cite{marden}, pp. 377-378).   

Here we are interested only in knots $K \subset\Sigma \times I$. If $\Sigma=S^1 \times S^1$, a knot with simple and Haken exterior will be hyperbolic of finite volume. If $\Sigma \ne S^1 \times S^1$ or $S^2$, a knot with simple and Haken exterior is hyperbolic and the metric is totally geodesic on $\partial(\Sigma \times I \smallsetminus K)$ (Morgan \cite{morgan}, Theorem B$\,'$). Below we will prove hyperbolicity of knots in $\Sigma \times I$ by showing they have simple and Haken exterior. By the remarks of the preceding paragraph, such knots are of finite volume. For more on volumes of hyperbolic knots in thickened surfaces, the reader is referred to Adams et al. \cite{adams_2}.

We are now ready to make the following definition of a hyperbolic virtual knot. Observe that the case of classical knots is treated separately.   

\begin{definition}[Hyperbolic virtual knot] \label{defn_hyp} A virtual knot $\upsilon$ is said to be \emph{hyperbolic} if $\upsilon$ can be represented by some hyperbolic knot $K \subset \Sigma \times I$. If $K\subset S^2 \times I$ (i.e. $\upsilon$ is classical) we will take this to  mean that $K$ is hyperbolic in the copy of $S^3$ obtained by capping off $S^2 \times I$ with two $3$-balls attached along their boundaries to $\partial (S^2 \times I)$.
\end{definition}

An immediate consequence of the definition is that a classical knot is hyperbolic in $S^3$ if and only if it is a hyperbolic virtual knot. Indeed, if a non-hyperbolic classical knot has a hyperbolic representative $K \subset \Sigma \times I$, $\Sigma \ne S^2$, then any genus reducing destabilzation of $K$ serves as a incompressible annulus in the exterior of $K$ that is not $\partial$-parallel. As hyperbolic $K \subset \Sigma \times I$ do not admit such annuli, a classical knot cannot have a non-minimal hyperbolic representative. Other examples of hyperbolic virtual knots come from the recent work of Adams et al. \cite{adams} on alternating knots in thickened surfaces. Recall that a knot diagram on a surface is said to be alternating if the crossings alternate between over and over while traversing the diagram. A diagram is said to be \emph{fully alternating} if, in addition, the complement of its underlying immersed curve consists only of discs. By Adams et al. \cite{adams}, Theorem 1, every locally trivial fully alternating knot in $\Sigma \times I$, $\Sigma \ne S^2$, is hyperbolic in $\Sigma \times I$ (see \cite{adams}, Theorem 1). For virtual knots, we have the following application.

\begin{proposition} \label{prop_adams} Suppose $\upsilon$ is an alternating diagram of a non-classical virtual knot. Let $\Sigma$ be the Carter surface of $\upsilon$ and $K \subset \Sigma \times I$ the knot corresponding to the diagram $\upsilon$. If $K$ is locally trivial in $\Sigma \times I$, then $\upsilon$ is hyperbolic.
\end{proposition}

\begin{proof} Since $\upsilon$ is non-classical, $\Sigma \ne S^2$. The Carter surface is obtained from $\upsilon$ by a handle decomposition: 0-handles at the crossings, 1-handles along the arcs of the diagram, and 2-handles along the boundary components (see Section \ref{sec_review_virtual}). Thus, the complement of the immersed curve underlying $\upsilon$ on $\Sigma$ consists only of discs (in particular, the 2-handles). It follows that $K$ is fully alternating. Since $K$ is also locally trivial, it is hyperbolic by \cite{adams}, Theorem 1.
\end{proof}

\begin{figure}[htb]
\begin{tabular}{c} \def\svgwidth{2in}
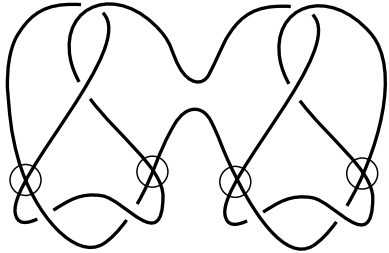 \end{tabular}
\caption{A hyperbolic virtual knot (4.105$\#$4.105) that is not a prime virtual knot.} \label{fig_hyp_not_prime}
\end{figure} 

As is well-known, every hyperbolic knot in $S^3$ is locally trivial and thus is prime as a classical knot. By Theorem \ref{thm_class_virt_prime}, a virtual knot that is equivalent to a hyperbolic classical knot is a prime virtual knot. Not every hyperbolic virtual knot, however, is prime.

\begin{example}\label{example_hyp_not_prime} Consider the prime  alternating virtual knot $\upsilon$=4.105. Let $\upsilon \#\upsilon$ be the connected sum depicted in Figure \ref{fig_hyp_not_prime}. Let $K\subset \Sigma \times I$ be a representative of $\upsilon$ on its Carter surface and let $K\#K \subset (\Sigma \#\Sigma) \times I$ represent $\upsilon\#\upsilon$. Since prime decompositions of virtual knots have unique prime summands \cite{matveev_roots} and $K\#K$ has no non-trivial classical summands, $K\#K$ contains no non-trivial local knots. It follows that $K\# K$ is locally trivial in $(\Sigma \#\Sigma) \times I$. Note also that the local triviality of $K\# K$ can be proved using Adams et al. \cite{adams}, Theorem 2. Thus $\upsilon \# \upsilon$ is hyperbolic but not prime. 
\end{example}

As in the case of classical knots, every virtual knot is concordant to a hyperbolic virtual knot. This follows from the main theorem of Myers \cite{myers} (Theorem 7.1), as we will now describe. Let $M$ be a compact oriented $3$-manifold whose boundary contains no $2$-spheres. In \cite{myers}, knots $K_0,K_1\subset M$ are said to be \emph{concordant in $M$} if they cobound a properly and smoothly embedded annulus in $M \times I$, where $K_i \subset M \times \{i\}$ for $i=0,1$. Myers proved that every knot in $M$ is concordant in $M$ to a hyperbolic knot. Applying this to the case of knots in $\Sigma \times I$, $\Sigma \ne S^2$, it follows that every knot in $\Sigma \times I$ is concordant in $\Sigma \times I$ to a hyperbolic knot. Now, if two knots in $\Sigma \times I$ are concordant in $M=\Sigma \times I$, then they represent concordant virtual knots\footnote{The converse of this statement is false; see e.g. \cite{bbc}.}. Indeed, just set $W=\Sigma\times I$ in the definition of concordance for knots in thickened surfaces (see Section \ref{sec_motivate}). Thus, every non-classical virtual knot is concordant to a hyperbolic virtual knot. Furthermore, Myers' Theorem 7.1 applied to the case of $M=S^3$ implies that every classical knot is concordant to a hyperbolic knot.

Putting this all together, it follows that every virtual knot is concordant to a hyperbolic virtual knot, but not all hyperbolic virtual knots are prime. Theorem \ref{thm_B} shows that every virtual concordance class contains a virtual knot that is simultaneously prime and hyperbolic. An example of such a concordance is given in Figure \ref{fig_conc_example}. It is a concordance from a composite knot $4.99\# 4.105$ to the prime knot $4.105$. The virtual knot 4.105 has virtual genus one, so by Theorem \ref{thm_prime_implies_loc_triv}, any genus one representative is locally trivial. It is easy to see that the Carter surface of the given diagram of 4.105 has genus one. Proposition \ref{prop_adams} then implies that $4.105$ is both hyperbolic and prime.

\subsection{Tangles in $3$-balls} \label{sec_tangles} Here we review tangles in $3$-balls and their geometric properties. First recall that a \emph{$3$-manifold pair} $[M,F]$ is a $3$-manifold $M$ and a surface $F \subset \partial M$. A pair $[M,F]$ is said to be \emph{irreducible} if $M$ is itself irreducible and $F$ is incompressible in $M$.

A ($2$-string) \emph{tangle} (see Lickorish \cite{lick}) is a pair $(B,t)$ where $B$ is a $3$-ball and $t=t_1 \sqcup t_2$ is a pair of disjoint arcs embedded in $B$ such that $t \cap \partial B=\partial t$. Tangles are considered equivalent up to homeomorphism of pairs $(B,t) \to (B',t')$. The (2-string) \emph{untangle} is the tangle $(B^2 \times I,\{a,b\} \times I)$, where $a,b \in \mathring{B}^2$, $a \ne b$. See Figure \ref{fig_tangles} (left and center). Now, let $N(t_1)$, $N(t_2)$ be disjoint regular neighborhoods of $t_1$, $t_2$ that intersect $\partial B$ in four discs $D_1, D_2, D_3, D_4$. Let $E(t)=\overline{B \smallsetminus (N(t_1) \cup N(t_2))}$ and $F(t)=\partial B \smallsetminus (\mathring{D}_1 \cup \mathring{D}_2 \cup \mathring{D}_3 \cup \mathring{D}_4)$. We will call $E(t)$ the \emph{exterior} of $(B,t)$ and the $3$-manifold pair $[E(t),F(t)]$ the \emph{exterior pair} of $(B,t)$.
\newline

\begin{figure}[htb]
\begin{tabular}{ccc} \begin{tabular}{c} \def\svgwidth{1.5in}
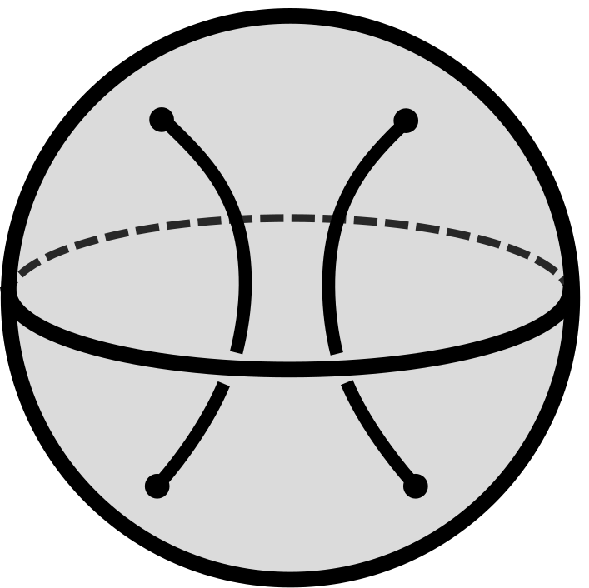 \\ an untangle \end{tabular} & 
\begin{tabular}{c} \def\svgwidth{1.5in}
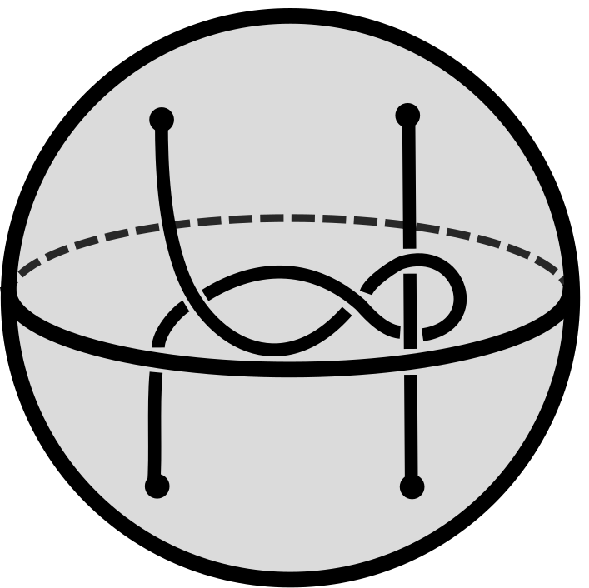 \\ an untangle \end{tabular} 
& \begin{tabular}{c} \def\svgwidth{1.5in}
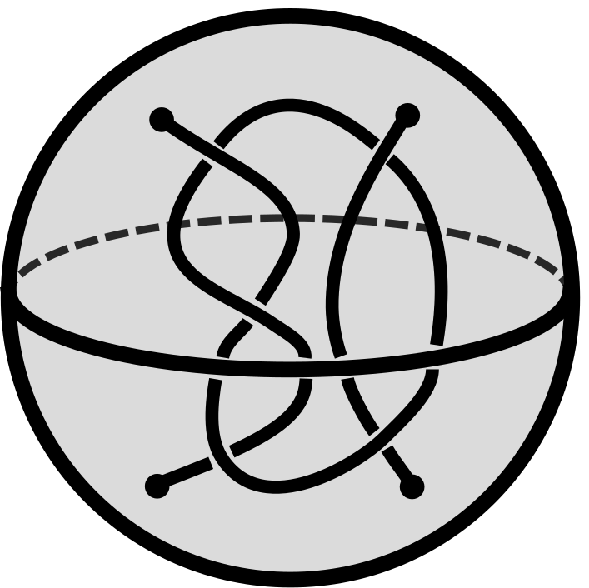 \\ a prime tangle \end{tabular}
\end{tabular}
\caption{Some (2-string) tangles in $3$-balls.} \label{fig_tangles}
\end{figure}

A tangle $(B,t)$ is said to be \emph{prime} if (1) both arcs of $t$ are locally trivial in $B$ and (2) $(B,t)$ is not equivalent to the untangle (see \cite{lick}). It is often convenient to replace condition (2) with: $(2')$ no disc properly embedded in $B$ and disjoint from $t$ can separate the arcs of $t$. If a tangle is prime, then its exterior pair is irreducible.

In \cite{lick}, Theorem 1, Lickorish proved that if a (one or two component) link $L \subset S^3$ can be decomposed into a union of two prime tangles $(A,s)$, $(B,t)$, where $A \cup B=S^3$, $A \cap B$ is a $2$-sphere intersecting $L$ transversely in the four points $s \cap t$, and $L=s \cup t$, then $L$ is a prime link (i.e. locally trivial in $S^3$).  

A tangle $(B,t)$ is said to be \emph{atoroidal} (\emph{anannular}) if $E(t)$ is atoroidal (respectively, anannular). In particular, an atoroidal 2-string tangle admits no incompressible tori, since $\partial E(t)$ is a two-holed torus. In \cite{soma}, Theorem 1, Soma proved that a one or two component link decomposed as above into two prime and atoroidal tangles is non-split, atoroidal, and prime in $S^3$. 

\begin{remark} A remark on terminology is in order. In \cite{kanenobu,myers_0,myers}, an ``atoroidal'' tangle is defined as a tangle whose exterior is a simple $3$-manifold, and hence is both atoroidal and anannular in our sense. In \cite{kanenobu}, a ``simple'' tangle is defined as a prime tangle that contains no incompressible torus. There are examples of ``simple'' tangles that are not ``atoroidal''.  Our terminology has been chosen so that the word ``simple'' is not multiply defined.
\end{remark}

Myers' method of constructing hyperbolic knots in compact $3$-manifolds generalizes the tangle method for knots in $S^3$. This strategy will be employed in the proof of Theorem \ref{thm_B} to construct knots in $\Sigma \times I$ with simple and Haken exterior. We take a moment to review it here as it will be used in the proof of Theorem \ref{thm_B}. Suppose that $M$ is a $3$-manifold obtained from the pairs $[M_0,F]$, $[M_1,F]$ by gluing along $F$, so that $M=M_0 \cup M_1$ and $F=M_0 \cap M_1=\partial M_0 \cap \partial M_1$. To identify sufficient conditions on $[M_i,F]$ so that $M$ is simple and Haken, we recall the following properties:

\begin{itemize}
\item\underline{Property A:} $[M,F]$ and $[M,\overline{\partial M \smallsetminus F}]$ are irreducible $3$-manifold pairs, no component of $F$ is a disc or $S^2$, and any disc $D$ properly embedded in $M$ such that $D \cap F$ is an arc is necessarily $\partial$-parallel.

\item\underline{Property B$\,'$:} $[M,F]$ has Property A, no component of $F$ is an annulus or torus, every incompressible annulus $A$ in $M$ with $\partial A \cap \partial F=\emptyset$ is $\partial$-parallel, and every incompressible torus in $M$ is $\partial$-parallel.

\item\underline{Property C$\,'$:} $[M,F]$ has Property B$\,'$, and every disc $D$ in $M$ such that $D \cap F$ is a pair of disjoint arcs is necessarily $\partial$-parallel.
\end{itemize}

The following lemma gives sufficient conditions on $[M_0,F]$, $[M_1,F]$ so that $M$ is a simple Haken $3$-manifold.

\begin{lemma}[Myers \cite{myers}, Lemma 2.5] \label{lemma_myers_2p5} If $[M_0,F]$ has property B$\,'$ and $[M_1,F]$ has property C$\,'$, then $M$ is simple and Haken. In particular, if $M_0,M_1$ are simple and Haken and $F$, $\overline{\partial M_0 \smallsetminus F}$ and $\overline{\partial M_1 \smallsetminus F}$ are incompressible, and no component of $F$ is a disc, $2$-sphere, annulus, or torus, then $M$ is simple and Haken.
\end{lemma}

Examples of manifold pairs having property C$\,'$ are given in the the next lemma.

\begin{lemma}[Kanenobu \cite{kanenobu}, Theorem 2] \label{lemma_kanenobu} The exterior pair of a prime atoroidal (2-string) tangle has property C$\,'$.
\end{lemma}

\begin{example} \label{example_KT} The 2-string tangle $KT$ depicted in Figure \ref{fig_kt} is called the \emph{Kinoshita-Terasaka tangle}. In \cite{bleiler}, Lemma 2.1, Bleiler proved that $KT$ is prime. In \cite{soma}, Lemma 3, Soma proved that $KT$ is atoroidal. By Lemma \ref{lemma_kanenobu}, the exterior pair of $KT$ has Property C$\,'$. On the other hand, $KT$ is not anannular (e.g. see Kawauchi \cite{kawauchi_KT}, Remark 5.6).   
\end{example}

\begin{figure}[htb]
\begin{tabular}{c} 
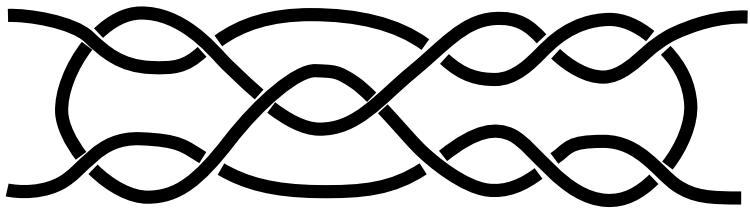 \end{tabular}
\caption{The Kinoshita-Terasaka tangle; a prime atoroidal tangle.} \label{fig_kt}
\end{figure}

\subsection{Complementary tangles} \label{sec_complementary} As discussed in the previous section, classical knots in $S^3$ with useful properties can be constructed by gluing together tangles in $3$-balls. Here we will construct knots in thickened surfaces with useful properties by gluing tangles in $3$-balls to \emph{complementary tangles}. Let $B$ be a $3$-ball embedded in the interior of a thickened surface $\Sigma \times I$. Let $E_B=\Sigma\times I \smallsetminus \mathring{B}$ be the exterior of $B$. A \emph{(2-string) complementary tangle} is a pair of disjoint properly embedded arcs $r=r_1\sqcup r_2$ such that $\partial r_1 \sqcup \partial r_2 \subset\partial B$. A complementary tangle is denoted by $(E_B,r)$. Complementary tangles will always be assumed to be oriented. Let $N(r)$ be a fixed regular neighborhood of $r$ in $E_B$ and define the \emph{exterior of} $r$ to be $E(r)=\overline{E_B \smallsetminus N(r)}$. Let $F(r)$ be the sphere with four holes given by $\overline{\partial B \smallsetminus \partial B \cap N(r)}$. The \emph{exterior pair} of $(E_B,r)$ is the $3$-manifold pair $[E(r),F(r)]$.

\begin{figure}[htb]
\begin{tabular}{c} \def\svgwidth{3.5in}
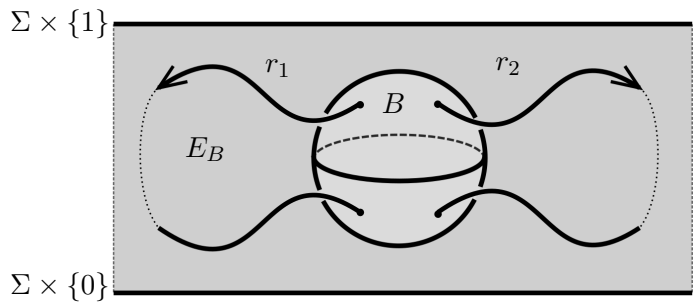 
\end{tabular}
\caption{A complementary tangle $(E_B,r)$ for a $3$-ball $B \subset \Sigma \times I$.} \label{fig_comp_tangle}
\end{figure}

Complementary tangles arise from knots and tangles as follows. Let $K \subset \Sigma \times I$ be an oriented knot and $B \subset \Sigma \times I$ a $3$-ball intersecting $K$ in two arcs $t=t_1 \sqcup t_2$ so that $K$ intersects $\partial B$ transversely. Then we say that $(B,t)$ \emph{intersects} $K \subset \Sigma \times I$. Let $r=K \cap E_B$ denote the two arcs $r_1$, $r_2$ of $K$ in $E_B$. The \emph{complement to $(B,t)$ in $K$} is $(E_B,r)$. This will be denoted by $K\ominus_B t$. Observe that the exterior $E_K$ of $K$ can be identified with the union of the exteriors of a tangle $(B,t)$ and its complement in $K$ by identifying them along the common subsurface $F(r)$ in their boundaries. 

Our first task is to find sufficient conditions on $(B,t)$ and $K\ominus_B t$ to ensure that $K$ is locally trivial. To that end, we make the following definition.

\begin{definition}[Irreducible, locally trivial complementary tangle] Let $B$ be a $3$-ball embedded in the interior of $\Sigma \times I$. A complementary tangle $(E_B,r)$ will be called \emph{irreducible} if its exterior pair $[E(r),F(r)]$ is irreducible. It will be called \emph{locally trivial} if each arc of $r$ is locally trivial in $E_B$ i.e., every splitting-$S^2$ intersecting $r$ transversely in two points bounds a $3$-ball $B'\subset E_B$ such that $B' \cap r$ is an unknotted arc.
\end{definition}

\begin{lemma} \label{lemma_loc_triv} Suppose a tangle $(B,t)$ intersects a knot $K \subset \Sigma \times I$, $\Sigma\ne S^2$. If $t$ is prime and its complementary tangle $K\ominus_B t$ is irreducible and locally trivial, then $K$ is locally trivial. 
\end{lemma}

\begin{proof} Let $S$ be a splitting-$S^2$ for $K$. Let $[E,F]$ be the exterior pair of $K\ominus_B t$. Then $F=F(t)$. Furthermore, suppose that $S$ is chosen so that $S \cap F$ has the minimal number of connected components among all splitting-$2$-spheres $S'$ such that $(S,S \cap K)$ and $(S',S'\cap K)$ are isotopic as pairs.  Since $K\ominus_B t$ is irreducible, $[E,F]$ is an irreducible pair. Since $(B,t)$ is prime, $[E(t),F]$ is also an irreducible pair. Thus, $F$ is incompressible in $E(K)=E(t) \bigcup_F E$. 

Let $A$ be the annulus obtained by deleting small open disc neighborhoods of $S \cap K$ from $S$ that are disjoint from $F$. Note that $A$ is incompressible in the exterior of $K$. To see this, let $D$ be a compression disc. Then $\partial D$ is essential in $A$. Let $D'$ be a disc in $S$ such that $\partial D'=\partial D$ and $D'$ intersects $K$ in one point. This implies $D \cup D'$ is a $2$-sphere intersecting $K$ in one point, which is impossible as $\Sigma \times I$ is irreducible. Thus $A$ is incompressible in the exterior of $K$.

Now observe that each component of $C\subset S \cap F$ must be essential in both $A$ and $F$. Indeed, since $F$ is incompressible in $E(K)$, an innermost inessential component $C$ in $A$ bounds a compression disc for $F$. Hence, $C$ must likewise be inessential on $F$. Since $C$ is inessential on both $A$ and $F$ and $\Sigma \times I$ is irreducible, it may as usual be removed by an isotopy. As this contradicts the hypothesis on $S$, there can be no components $C$ that are inessential in $A$. Similarly, it can be shown that there are no components $C\subset S \cap F$ that are inessential in $F$.

Let $C$ be an innermost component of $S \cap F$ on $S$. Then since $C$ is essential in $A$, $C$ bounds a disc $D$ in $S$ such that $D \cap K$ contains one point and $D \cap F=C$. Since $C$ is essential in $F$, $C$ must bound a disc $D' \subset \partial B$ containing one or two points of $K$. If $D'$ contains two points of $K$, then $D \cup D'$ is a 2-sphere intersecting $K$ in three points. Hence $D \cup D'$ is a non-separating $S^2$, which contradicts the fact that $\Sigma \times I$ is irreducible. Thus, $D'$ intersects $K$ in one point. This implies $D \cup D'$ is a splitting-$S^2$ of $K$ in $\Sigma \times I$. By irreducibility of $\Sigma \times I$, there is a  $3$-ball $B'' \subset \Sigma \times I$ with $\partial B''=D \cup D'$. If $D \subset B$, then we may assume by the local triviality of $(B,t)$ that $B''\subset B$ and $B''$ intersects $K$ in an unknotted arc. If $D \subset E_B$, then we may assume by the local triviality of $K\ominus_B t$ that $B''\subset E_B$ and that $B''$ intersects $K$ in an unknotted arc. Then $(S,S \cap K)$ may be modified by an isotopy to reduce the number of components of $S \cap F$, contradicting our choice of $S$.

Therefore, we may assume that $S \cap \partial B=\emptyset$. Then $S \subset B$ or $S\subset E_B$. Since $(B,t)$ and $K\ominus_B t$ are locally trivial, $S$ bounds a $3$-ball in either $B$ or $E_B$ that intersects $K$ in an unknotted arc. 
\end{proof}

Next we give sufficient conditions on a tangle $(B,t)$ and its complement $K \ominus_B t$ so that the knot $K$ has simple and Haken exterior.

\begin{lemma} \label{lemma_sum_hyperbolic} Suppose a tangle $(B,t)$ intersects a knot $K \subset \Sigma \times I$. If $t$ is prime and atoroidal and the exterior of $K\ominus_B t$ is simple and Haken, then the exterior of $K$ is simple and Haken.
\end{lemma}
\begin{proof} By Lemma \ref{lemma_kanenobu}, $(B,t)$ has Property C$\,'$.  By Lemma \ref{lemma_myers_2p5}, it is sufficient to prove that the exterior of $K\ominus_B t$ has Property B$\,'$. Let $[E,F]$ denote the exterior pair of $K\ominus_B t$. Since $F$ is a sphere with four holes, no component of $F$ is a disc, annulus, $2$-sphere, or torus. Since $E$ is simple and Haken, it is irreducible, boundary irreducible, anannular, and atoroidal.  This implies that every disc $D\subset E$ intersecting $F$ in a single arc is $\partial$-parallel. 

Next we show that $[E,F]$ is an irreducible pair. Note that $E$ has three boundary components. They are $\Sigma \times \{1\}$, $\Sigma \times \{0\}$ and a component $X$ homeomorphic to a two-holed torus. Since $E$ is $\partial$-irreducible, $\partial E$ is incompressible and hence the inclusion map induces on the fundamental group of each component of $\partial E$ an injection into $\pi_1(E,*)$ (for $* \in F$). Now, $F$ is sphere with four holes contained in $X$. Observe that $X$ is obtained from $F$ by attaching two annuli $A_1,A_2$ to $\partial F$. Then every component of $\partial F$ is essential in $X$. It follows that the inclusion map induces an injection $\pi_1(F,*) \to \pi_1(X,*) \to \pi_1(E,*)$ and we conclude that $F$ is incompressible. This implies that $[E,F]$ is an irreducible pair.

It remains only to prove that $[E,\overline{\partial E\smallsetminus F}]$ is an irreducible pair. Note that $\overline{\partial E\smallsetminus F}= \Sigma \times \{0\} \cup \Sigma \times \{1\} \cup \overline{X\smallsetminus F}$. The first two components are incompressible in $E$. Note that $\overline{X \smallsetminus F}$ is the pair of annuli $A_1,A_2$. Since each $A_i$ has fundamental group generated by an essential closed curve in $X$, it follows that the inclusion $\pi_1(A_i,*) \to \pi_1(X,*) \to \pi_1(E,*)$ (for $*\in A_i$) is injective and thus $\overline{X\smallsetminus F}$ is incompressible.
\end{proof}

Lastly, we have a technical lemma that gives sufficient conditions for a complementary tangle to be prime. It will be used in the proof of Theorem \ref{thm_B}.

\begin{lemma} \label{lemma_comp_prime} Suppose a tangle $(B,t)$ intersects a knot $K \subset \Sigma \times I$, $\Sigma \ne S^2$. If $K$ is locally trivial and the exterior of $K \ominus_B t$ is irreducible, boundary irreducible, and atoroidal, then $K\ominus_B t$ is an irreducible and locally trivial complementary tangle.
\end{lemma}

\begin{proof} Let $[E,F]$ be the exterior pair of $K\ominus_B t$. Again, $E$ has three boundary components $\Sigma \times \{1\}$, $\Sigma \times \{0\}$, and a component $X$ homeomorphic to a two-holed torus. By the proof of Lemma \ref{lemma_sum_hyperbolic}, $[E,F]$ is irreducible. Thus it needs only to be shown that $K \ominus_B t$ is locally trivial. 

Let $S$ be a splitting-$S^2$ of $K\ominus_B t$ and let $A$ be the annulus obtained by deleting from $S$ small neighborhoods of $K \cap S$. Since $\Sigma \times I$ is irreducible, $S$ bounds a $3$-ball $B'' \subset \Sigma \times I$. Since $E$ is also irreducible, either $B \subset B''$ or $B'' \subset E_B$. If $B'' \subset E_B$, then the local triviality of $K$ implies that $B''$ intersects $K$ in an unknotted arc.

We complete the proof by showing that $B$ cannot be in $B''$. If $B \subset B''$, let $E_{B''}=\Sigma \times I \smallsetminus \mathring B''$ and define $r''$ to be $K \cap E_{B''}$. Note that $r''$ is a single arc. Let $M$ be the solid torus formed from the union of a regular neighborhood of $r''$ and $B''$. Set $T=\partial M$. Observe that $X \subset M$.

Since $E$ is atoroidal, $T$ is either compressible or $\partial$-parallel in $E$. The torus $T$ cannot be $\partial$-parallel to the two-holed torus $X$. If $T$ is $\partial$-parallel to $\Sigma \times \{1\}$ or $\Sigma \times \{0\}$ in $E$, then $T$ is incompressible in $\Sigma \times I$. But this is impossible as $T$ bounds the solid torus $M \subset \Sigma \times I$. If $T$ is compressible in $E$, let $D$ be a compression disc. The 2-sphere created by compressing $T$ along $D$ separates $E$ into a component that contains $X$ and a component that contains $\partial (\Sigma \times I)$. Hence the sphere doesn't bound a $3-$ball. Since $E$ is irreducible, this is impossible and $B$ cannot be contained in $B''$. 
\end{proof}

\subsection{Tangle operations and prime virtual knots} \label{sec_tangle_ops} Here we give the main construction used in the proof of Theorem \ref{thm_B}. The join of a tangle and a complementary tangle is defined and then studied in terms of the virtual genus, local triviality, compressible decompositions, and special decompositions.

\begin{definition}[Join of a tangle and a complementary tangle] \label{defn_tangle_join} Let $B$ be a $3$-ball in $\Sigma \times I$ and let $(E_B,r)$ be a complementary tangle, with $r=r_1\sqcup r_2$. Let $(B',t')$ be a $2$-tangle with oriented arcs, where $t'=t_1' \sqcup t_2'$. Identify $B'$ and $B$ by an orientation preserving diffeomorphism such that $\partial t_i'=\partial r_i$ and terminal/initial points of $\partial t_i$ are mapped to initial/terminal points of $\partial r_i$, respectively, for $i=1,2$. Furthermore suppose that $K=r \cup t' \subset \Sigma \times I$ is a knot in $\Sigma \times I$. Then $K$ is called \emph{a join of the tangle $(B',t')$ and the complementary tangle $(E_B,r)$}. A join of $(B',t')$ and $(E_B,r)$ will be denoted by $r\oplus_{B'} t'$.\end{definition}

If a tangle $(B,t)$ intersects a knot $K \subset \Sigma \times I$ and $(B',t')$ is any other tangle, new knots $K'\subset \Sigma \times I$ may be obtained using complements and joins. The new knots are of the form $(K\ominus_B t) \oplus_{B'} t'$. The following lemma gives sufficient conditions on $(B',t')$ and $K\ominus_B t$ so that $K$ and $K'$ have the same virtual genus.

\begin{lemma} \label{lemma_min_genus_preserved} Suppose a tangle $(B,t)$ intersects a knot $K \subset \Sigma \times I$, $\Sigma \ne S^2$, and that $K\ominus_B t$ is an irreducible complementary tangle. Let $(B',t')$ be a prime tangle. If $K$ is a minimal representative of a virtual knot, then $K'=(K\ominus_B t)\oplus_{B'} t'$ is a minimal representative of a virtual knot.
\end{lemma}

\begin{proof} Let $[E,F]$ be the exterior pair of $K\ominus_B t$ and let $E(K')$ be the exterior of $K'$. Since $[E,F]$ is irreducible and $[E(t'),F]$ is irreducible, $F$ is incompressible in $E$, $E(t')$ and thus $E(K')$. Let $A$ be an incompressible destabilization annulus for $K'$. If $A \cap F=\emptyset$, then $A$ is a destabilization for $K$. Since $K$ is minimal, Theorem \ref{thm_compress} implies the contradiction that $A$ is compressible in $\Sigma \times I$. Assume then that $A \cap F \ne \emptyset$ and that $A$ is chosen to minimize the number of components of $A \cap F$, up to isotopy. 

Since $A$ and $F$ are both incompressible, an innermost inessential component $C \subset A \cap F$ on one of them must also be inessential on the other. As usual, such $C$ can be removed by isotopy, which contradicts the hypothesis on $A$. Hence, we may assume that $A \cap F$ contains no components that are inessential on $A$.

Thus, all $C \subset A\cap F$ are $\partial$-parallel in $A$. Choose $C$ to be outermost, so that $C$ cobounds an annulus $\alpha \subset A$ with a component of $\partial A$ such that $\alpha \cap F=C$. Since $A$ is incompressible, Lemma \ref{lemma_experiment_2} implies that $C$ is essential on $\partial B'$. Since $2$-spheres have no  essential curves, this is a contradiction. This implies that there are no incompressible destabilizations of $K'$ in $\Sigma \times I$. By Theorem \ref{thm_compress}, $K'$ is minimal. 
\end{proof}

The next lemma gives sufficient conditions on a join of $(B',t')$ and $K\ominus_B t$ so that $(K\ominus_B t)\oplus_{B'} t'$ represents a prime virtual knot.

\begin{lemma}\label{lemma_no_annular} Suppose a tangle $(B,t)$ intersects a minimal representative $K \subset \Sigma \times I$ of a non-classical virtual knot. Furthermore, suppose that $K \ominus_B t$ is irreducible and locally trivial and that $(B',t')$ is prime. Then it follows that:
\begin{enumerate}
\item any knot $K'=(K\ominus_B t)\oplus_{B'} t'$ is locally trivial, 
\item if $K$ admits no incompressible decomposing annuli, then neither does $K'$, and
\item if $K$ admits no special decompositions, then neither does $K'$.
\end{enumerate}
\end{lemma}

\begin{remark} Some paragraphs below are marked with daggers $(\dagger\cdots\dagger)$ for future reference.
\end{remark}

\begin{proof}[Proof of Lemma \ref{lemma_no_annular} (1)] As $K$ is non-classical, $\Sigma \ne S^2$. Since $(B',t')$ is prime and its complementary tangle $K\ominus_B t$ in $(K\ominus_B t)\oplus_{B'} t'$ is both irreducible and locally trivial, Lemma \ref{lemma_loc_triv} implies that $(K\ominus_B t)\oplus_{B'} t'$ is locally trivial.
\end{proof}

\begin{proof}[Proof of Lemma \ref{lemma_no_annular} (2)] Let $[E,F]$ be the manifold pair of $r=K\ominus_B t$ and let $E(K')=E(r) \cup_F E(t')$. Since $(B',t')$ is prime and $K\ominus_B t$ is irreducible, $F$ is incompressible in $E(t')$, $E=E(r)$, and $E(K')$. Let $A$ be an decomposition annulus for $K'$. Suppose that $A$ is incompressible in $\Sigma \times I$. We may assume $A \cap \partial B=A \cap F$. Suppose $A$ is chosen to minimize the number of components of $A \cap F$ among all annuli $A'$ with $(A,A \cap K')$ isotopic to $(A',A' \cap K')$. If $A \cap F=\emptyset$, then $A$ is a decomposing annulus for $K$. Since $K$ admits no incompressible decomposing annuli, this is a contradiction. Assume that $A \cap F \ne \emptyset$. 

Let $Y$ be the space obtained by removing from $A$ small open neighborhoods of $A \cap K'$ that are disjoint from $F$. Then $\partial Y$ consists of four circles. As has been previously observed (e.g. Lemma \ref{lemma_loc_triv}), since $A$ is incompressible in $\Sigma \times I$ and $F$ is incompressible in $E(K')$, any component $C \subset A \cap F$ that is inessential and innermost in $Y$ can be removed by isotopy. As this contradicts the hypothesis on $A$, $A \cap F$ can have no components that are inessential in $Y$.

$(\dagger)$ Suppose that $C$ is $\partial$-parallel in both $A$ and $Y$ (see Figure \ref{fig_lemma_lemma_2}). Then choose $C$ to be outermost, so that it cobounds an annulus $\alpha \subset A$ with a component of $\partial A$. Since $A \cap F$ has no components inessential in $A$, $\alpha \cap \partial B'=C$. By Lemma \ref{lemma_min_genus_preserved}, $K'$ is minimal. Since $A$ is incompressible, Lemma \ref{lemma_experiment_2} implies (as in the proof of Lemma \ref{lemma_min_genus_preserved}) the contradiction that $C$ is an essential curve on $\partial B'=S^2$. Thus there can be no $C$ that are $\partial$-parallel in both $A$ and $Y$.

$(\dagger\dagger)$ Suppose that $C$ is $\partial$-parallel in $Y$ but not $\partial$-parallel in $A$. Then $C$ is inessential in $A$ and $C$ bounds a disc $D \subset A$ intersecting $K'$ in one point. Take $C$ to be innermost in $A$. Note that $C$ must be essential in $F$ for otherwise $C$ would bound a disc $D'$ in $F$ and the $2$-sphere $D \cup D'$ in $\Sigma \times I$ intersects $K'$ in only one point. This contradicts the fact that $\Sigma\times I$ is irreducible. Now, since $C$ is essential in $F$, it bounds a disc $D'$ in $\partial B'$ intersecting $K'$ in one or two points. As argued in Lemma \ref{lemma_loc_triv}, the irreducibility of $\Sigma \times I$ forbids the $2$-sphere $D \cup D'$ from intersecting $K'$ in 3 points. Hence, $D'$ cannot intersect $K'$ in two points. Then $D'$ intersects $K'$ once and it follows that $D \cup D'$ is a splitting-$S^2$ for $K'$. Since $K'$ is locally trivial, such curves $C$ may be removed by isotopy. This contradicts the minimality of $A$.

Suppose that $C$ is $\partial$-parallel in neither $Y$ nor $A$. Then $C$ bounds a disc $D$ in $A$ intersecting $K'$ in two points. By the previous considerations, if there is one such $C$, all components of $A \cap F$ must be of this type. Suppose $C$ is outermost in $A$. Note that $C$ bounds a disc $D' \subset \partial B'$ that intersects $K'$ in zero, one, or two points. Consider the annulus $A'=(A \smallsetminus D) \cup D'$. If $D'\cap K'=\emptyset$, then $A$ is a destabilizing annulus for $K$. Since $K\subset \Sigma \times I$ is a minimal representative, $A'$ must be compressible. Thus, each component of $\partial A=\partial A'$ must bound a disc in $\partial (\Sigma \times I)$. This implies the contradiction that $A$ is compressible. If $D'$ intersects $K'$ once, then $A'$ and hence $A$ cannot be separating in $\Sigma \times I$. If $D'$ intersects $K'$ twice, then $A'$ is a decomposition of $K$. Since $K$ admits no annular decompositions along incompressible annuli, $A'$ and hence $A$ must be compressible.

Lastly, suppose that $C$ is $\partial$-parallel in $A$ but not in $Y$. By the above, if there is such a $C$ then all components of $A \cap F$ must be of this type. This case can be shown to be impossible using the same argument as in paragraph $(\dagger)$. Thus $A \cap F=\emptyset$ and it follows as before that $K'$ admits no incompressible decomposing annuli.
\end{proof}

\begin{proof}[Proof of Lemma \ref{lemma_no_annular} (3)] Let $F,E(r),E(t'), E(K')$ be as in the proof of Lemma \ref{lemma_no_annular} (2). As discussed in the proof of Lemma \ref{lemma_no_annular} (2), $F$ is incompressible in $E(K')$. Let $A=A_1 \sqcup A_2$ be a special decomposition of $K'$. If $A \cap \partial B'=\emptyset$, then $A$ is a special decomposition of $K$. This possibility is prohibited by hypothesis. It will be argued that the general case may always be reduced to the case that $A \cap \partial B'=\emptyset$. By Lemma \ref{lemma_special_incompress}, $A_1$ and $A_2$ must be incompressible in $\Sigma \times I$. Suppose that $A$ is chosen to minimize the number of connected components of $A \cap F$ over all special decompositions of $K'$ isotopic to $(A,A \cap K')$. 

Since each $A_i$ is incompressible, the various types of components $C \subset A \cap \partial B'$ may be eliminated as in the proof of Lemma \ref{lemma_no_annular} (2). For $i=1,2$, let $Y_i$ denote the space obtained by deleting small open disc neighborhoods of $A_i \cap K'$ from $A_i$ that are disjoint from $F$. Define $Y=Y_1 \cup Y_2$. Since $A_i$ is incompressible in $\Sigma \times I$ and $F$ is incompressible in $E(K')$, an innermost circle argument may be used, as usual, to show that there are no components of $A \cap F$ that are inessential in some $Y_i$. By an argument similar to that of paragraph $(\dagger)$ there can be no component $C \subset A \cap F$ that is $\partial$-parallel in both $A_i$ and $Y_i$ for some $i$. Furthermore, there can be no components $C \subset A \cap F$ that are $\partial$-parallel in $Y$ but not in $A$. This follows exactly as in paragraph $(\dagger\dagger)$. Thus, the general case reduces to the case that $A\cap F=A\cap \partial B'=\emptyset$ and $A$ is a special decomposition of $K$. Therefore, $K'$ admits no special decomopositions.
\end{proof}

\subsection{Proof of Theorem \ref{thm_B}} \label{sec_proof_2} If $\upsilon$ is classical, let $K \subset S^3$ be a classical knot equivalent to $\upsilon$. By Myer's theorem \cite{myers}, $K$ is concordant to a hyperbolic knot $K'' \subset S^3$. Then $K''$ is locally trivial in $S^3$ and hence corresponds to a prime virtual knot in $S^2 \times I$. Let $\upsilon''$ be the virtual knot corresponding to $K''$. Then $\upsilon''$ and $\upsilon$ are also concordant as virtual knots and $\upsilon''$ is a hyperbolic virtual knot by Definition \ref{defn_hyp}. 

\begin{figure}[htb]
\begin{tabular}{ccc} \begin{tabular}{c}\def\svgwidth{1.75in}
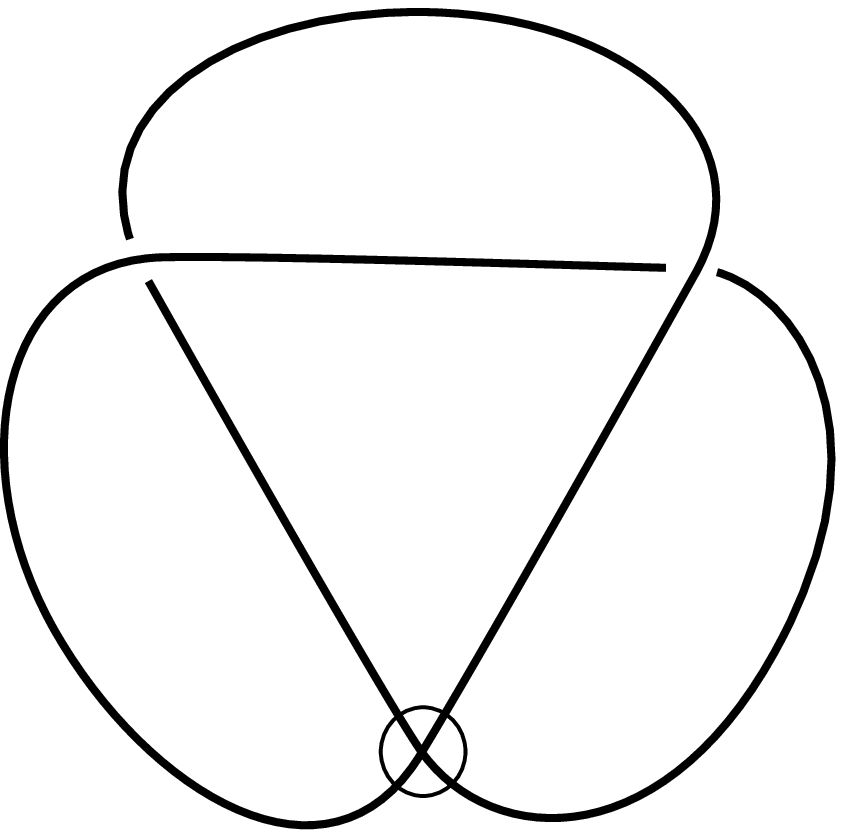 \end{tabular} &\begin{tabular}{c}\def\svgwidth{1.75in}
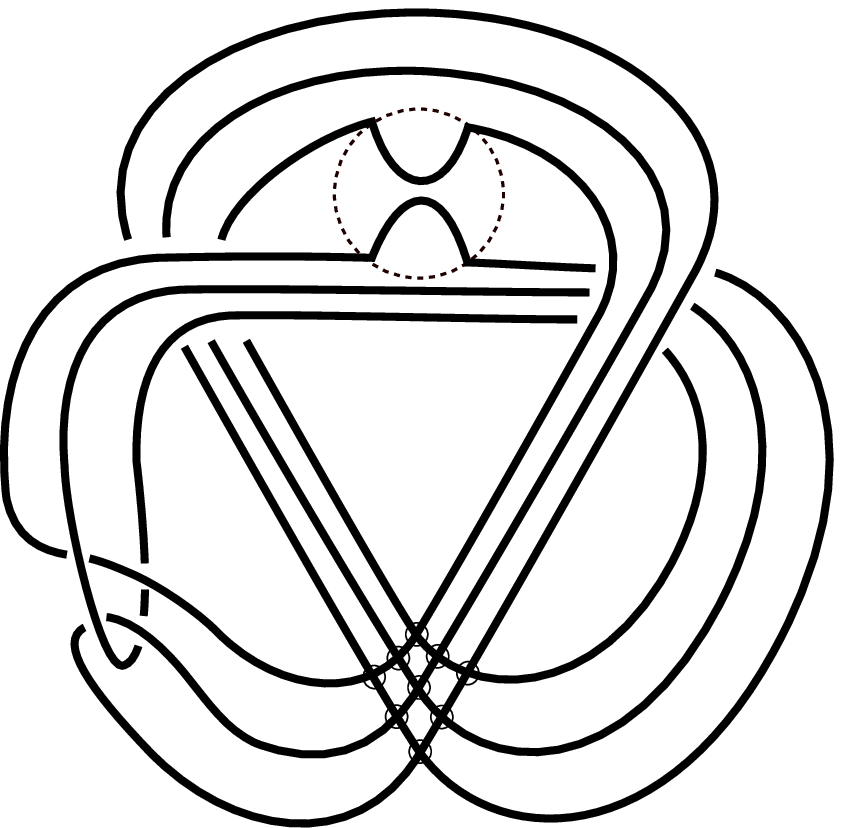 \end{tabular}
&
\begin{tabular}{c}\def\svgwidth{1.75in}
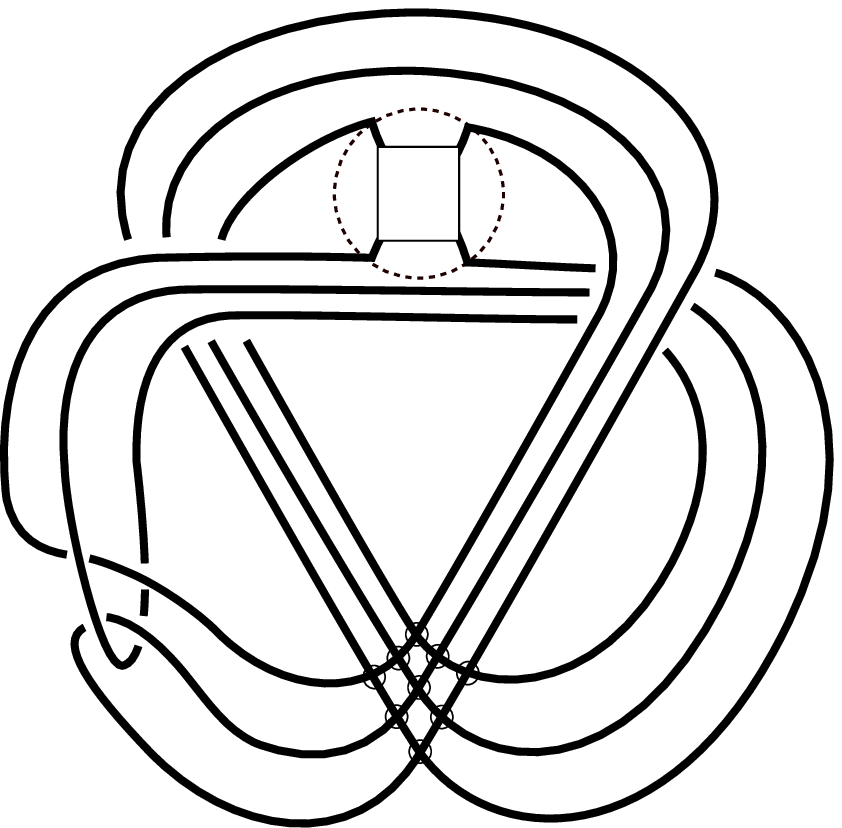 \end{tabular} \\
$\upsilon$ & $\upsilon'$ & $\upsilon''$
\end{tabular}
\caption{Schematic diagram for the proof of Theorem \ref{thm_B}.} \label{fig_outline}
\end{figure}

Suppose then that $\upsilon$ is non-classical and let $K \subset \Sigma \times I$ be a minimal representative of $\upsilon$. By Theorem \ref{thm_A}, $K$ is concordant to a satellite knot $K'\subset \Sigma \times I$ having the same virtual genus as $K$. The pattern $P$ may be assumed to have wrapping number greater than one. Then Lemma \ref{lemma_lemma} implies that $K'$ is locally trivial, admits only compressible annular decompositions, and admits no special decompositions. Theorem \ref{thm_B} will be proved by modifying $K'$ to a knot $K''\subset \Sigma \times I$ that is prime, hyperbolic, and has the same concordance class as $K$. The modification is illustrated schematically with virtual knot diagrams in Figure \ref{fig_outline}, where $K$, $K'$, and $K''$ are representatives of $\upsilon$, $\upsilon'$, and $\upsilon''$, respectively. The reason for the intermediate step from $K$ to $K'$ is so that $K''$ will admit no special decompositions. First it is shown there is a trivial tangle $(B,t)$ intersecting $K'$ such that $K'\ominus_B t$ is simple and Haken. To do this, we will use the following theorem of Myers.

\begin{lemma}[Myers \cite{myers}, Proposition 6.1] \label{thm_myers_tunnels} Let $M$ be a compact, connected, orientable, $3$-manifold such that $\partial M \ne \emptyset$ and contains no $2$-spheres. Let $Y$ be a component of $\partial M$. Then $M$ contains a properly embedded arc $\tilde{J}$ such that $\partial \tilde{J} \subset Y$ and the exterior of $\tilde{J}$ is a simple Haken manifold.
\end{lemma}

Let $E(K')$ be the exterior of $K'$ and $Y=\partial \overline{\Sigma \times I \smallsetminus E(K')}$. Let $t_1,t_2$ be disjoint arcs of $K'$ and let $N_1,N_2$ be regular neighborhoods of $t_1,t_2$ that intersect $Y$ in disjoint annuli $A_1,A_2$. Choose $\tilde{J}$ as in Lemma \ref{thm_myers_tunnels}, so that one point of $\partial \tilde{J}$ lies in $A_1$ and the other lies in $A_2$. Let $\tilde{N}$ be a regular neighborhood of $\tilde{J}$ intersecting each $A_i$ in a disc in $Y$. Let $B=N_1 \cup N_2 \cup \tilde{N}$ and $t=t_1 \sqcup t_2$. Then $(B,t)$ intersects $K'$ and the complementary tangle $K'\ominus_B t$ has simple and Haken exterior. To see that $(B,t)$ is trivial, observe that each arc $t_i \subset N_i$ is trivial and that $t_1$  is separated from $t_2$ by the disc $\tilde{N} \cap N_1$ (cf. \cite{myers}). 

Let $(B',t')$ be the Kinoshita-Terasaka tangle $KT$ (see Example \ref{example_KT}) and let $K''$ be a knot $(K'\ominus_B t)\oplus_{B'} t'$. Since $(B',t')$ is prime and atoroidal and the exterior of $K'\ominus_B t$ is simple and Haken, Lemma \ref{lemma_sum_hyperbolic} implies that $K''$ has simple and Haken exterior. Since $K'$ is locally trivial, Lemma \ref{lemma_comp_prime} implies $K'\ominus_B t$ is an irreducible and locally trivial complementary tangle. Since $K'$ is a minimal representative, Lemma \ref{lemma_min_genus_preserved} implies that $K''$ is also minimal.  By Lemma \ref{lemma_no_annular} (1) and (2), $K''$ is locally trivial and admits only compressible annular decompositions. Since $K'$ admits no special decompositions, Lemma \ref{lemma_no_annular} (3) implies that $K''$ admits no special decompositions. By Theorem \ref{thm_prime}, $K''$ represents a prime virtual knot $\upsilon''$. Since $K''$ has simple and Haken exterior, $\upsilon''$ has a hyperbolic representative and is thus a hyperbolic virtual knot. Lastly, it must be shown that $K''$ is concordant to $K'$. Since $(B,t)$ is trivial, it is sufficient to prove that $KT$ is concordant to the trivial $2$-string tangle. Figure \ref{fig_kt_movie} is a movie showing such a concordance.

\begin{figure}[htb]
\begin{tabular}{c} \def\svgwidth{4in}
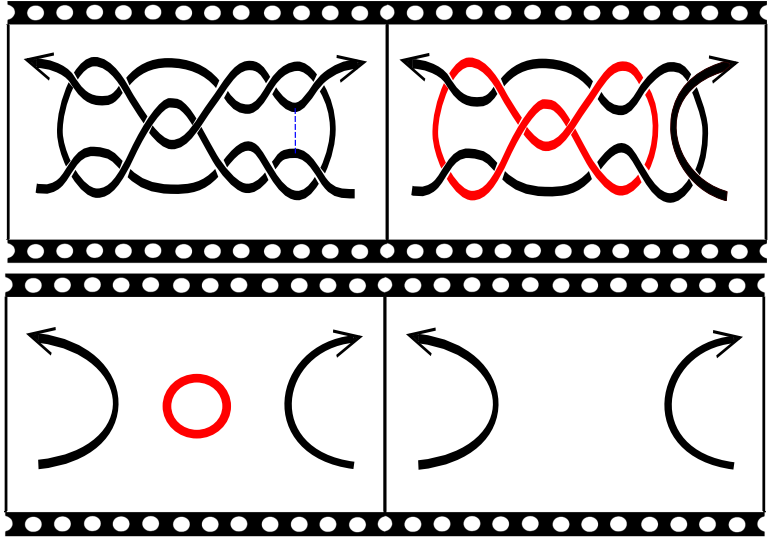 \end{tabular}
\caption{A concordance between the Kinoshita-Terasaka tangle and the trivial tangle.} \label{fig_kt_movie}
\end{figure}

\section{Preservation of the Alexander polynomial} \label{sec_alex}

This section proves Theorem \ref{thm_C}, that every almost classical knot $\upsilon$ is concordant to a prime satellite AC knot $\upsilon'$ and a prime hyperbolic AC knot $\upsilon''$, both of which have the same Alexander polynomial as $\upsilon$. Section \ref{sec_AC_review} reviews the essentials of AC knots. The proof of Theorem \ref{thm_C} takes place in Section \ref{sec_proof_complete}.

\subsection{Almost classical (AC) links} \label{sec_AC_review} A virtual link is said to be \emph{almost classical (AC)} if it has a representative $L \subset \Sigma \times I$ that is homologically trivial in $\Sigma \times I$. Almost classical knots were orginally defined by Silver-Williams in \cite{silwill} as virtual knots possessing an Alexander numerable diagram. This definition is equivalent to the one used here (Boden et al. \cite{acpaper}, Proposition 6.1). Another equivalent condition is that $K$ bounds a Seifert surface in some thickened surface $\Sigma \times I$. For a proof, and for other equivalent conditions defining AC knots and links, the reader is referred to \cite{acpaper}, Section 6. 

Since AC knots have representatives that bound Seifert surfaces, Alexander polynomials of AC knots may be defined in the same way as for classical knots. Let $F$ be a Seifert surface of genus $g$ for $L$ and let $\{a_1,\ldots,a_{2g}\}$ be a basis for $H_1(F)$. The $\pm$-Seifert matrices for $F$ are the $2g \times 2g$ matrices $V^{\pm}$ defined by $V^{\pm}=(\lk_{\Sigma}(a^{\pm}_i,a_i))$, where $a_i^{\pm}$ means the $\pm$-push-off of $a_i \subset F$. Unlike the classical case, $V^+$ and $V^-$ are generally not transposes of one another. Nonetheless, the Alexander polynomial can be defined as:
\[
\Delta_{L}(t)=\det(tV^{-}-V^+).
\]
In \cite{acpaper}, Corollary 7.3, Boden et al. proved that $\Delta_L(t)$ is an invariant of AC links that is well-defined up to multiplication by units in $\mathbb{Z}[t,t^{-1}]$.

The Alexander polynomial $\Delta_L(t)$ of an AC link $L$ is related to its group $G_L$ as follows. Recall that the group of the link is given by a presentation of the form:
\[
G_L=\left<g_1,\ldots,g_n \left| g_{i+1}=g_j^{\varepsilon_i} g_i g_j^{-\varepsilon_i}, i=1,\ldots,n\right.\right>.
\]
The presentation can be read off a virtual knot diagram or Gauss diagram, where the generators $g_i$ correspond to arcs between classical under-crossings and there is one relation at each crossing. In the special case of AC links, the first elementary ideal of the Alexander module of $G_L$ is principal and generated by $\Delta_L(t)$ (see \cite{acpaper}, Corollary 7.3).

The Alexander polynomial of AC links satisfies a skein relation. Suppose that $L_+,L_-,L_0$ is a triple of link diagrams on a surface $\Sigma$ representing homologically trivial links in $\Sigma \times I$ and that the diagrams differ only in small ball $B$ in $\Sigma$ as in Figure \ref{fig_triple}. Each of the knots bounds a Seifert surface in $\Sigma \times I$ and the triple of Seifert surfaces $F_+,F_-,F_0$ may be chosen so that they are identical outside of $B$ and inside of $B$ they appear as Figure \ref{fig_triple}. If the triple $F_+,F_-,F_0$ is used to compute $\Delta_{L_+}(t)$, $\Delta_{L_-}(t)$, $\Delta_{L_0}(t)$, then the usual skein relation is satisfied (see \cite{acpaper}, Theorem 7.11):

\begin{eqnarray}\label{skein}
\Delta_{L_+}(t)-\Delta_{L_-}(t)&=&(t-t^{-1})\Delta_{L_0}(t).
\end{eqnarray}

\begin{figure}[htb]
\begin{tabular}{ccc}   
\begin{tabular}{c} \def\svgwidth{.75in}
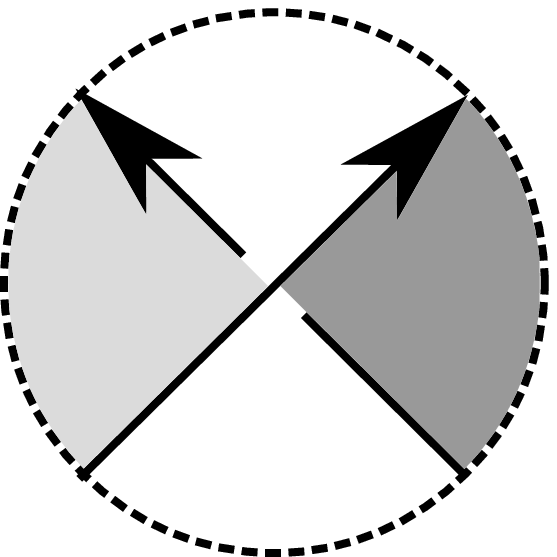 \\ $L_+=\partial F_+$ \end{tabular} & \begin{tabular}{c} \def\svgwidth{.75in}
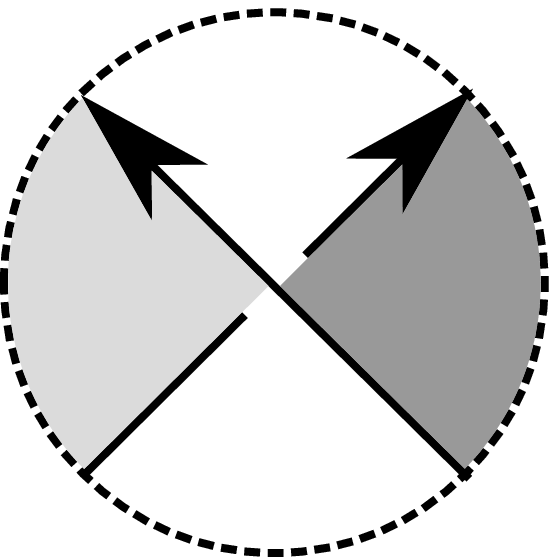 \\ $L_-=\partial F_-$ \end{tabular} & \begin{tabular}{c} \def\svgwidth{.75in}
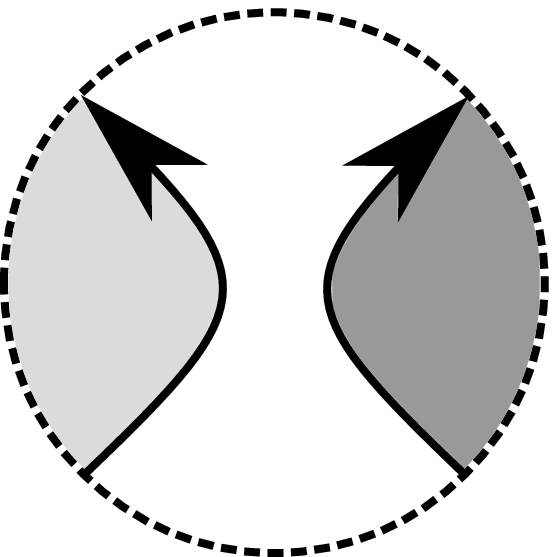 \\ $L_0=\partial F_0$ \end{tabular}
\end{tabular}
\caption{The triple of links $(L_+,L_-,L_0)$ and Seifert surfaces $(F_+,F_-,F_0)$.} \label{fig_triple}
\end{figure} 

A classical link where all the components bound disjoint Seifert surfaces is called a \emph{boundary link}. A link $L \subset \Sigma \times I$ will similarly be called a \emph{boundary link} if the components of $L$ bound pairwise disjoint Seifert surfaces in $\Sigma \times I$. If a virtual link has a representative in some $\Sigma \times I$ that is a boundary link in $\Sigma \times I$, then it will be called a \emph{boundary virtual link}. The set of boundary virtual links forms a subset of AC links.   

For a classical boundary $m$-component link $L$, the longitudes lie in the infinite intersection of the lower central series of the group $G_L$. This implies that the $(m-1)$-st elementary ideal of a boundary link is vanishing (see e.g. Hillman \cite{hill}). Since the lower central series quotients of a link group are concordance invariants (Stallings \cite{stallings_central}, Casson \cite{casson}), any link that is concordant to a boundary link also has vanishing $(m-1)$-st elementary ideal. In \cite{bbc}, this result was generalized to virtual links, where it was used to prove that the generalized Alexander polynomial of a slice virtual knot is vanishing. 

\begin{theorem}[see \cite{bbc}, Theorem 4.13] \label{lemma_bbc} If an $m$-component virtual link $\lambda$ is concordant to a boundary virtual link, then its $(m-1)$-st elementary ideal is trivial.
\end{theorem}

\subsection{Proof of Theorem \ref{thm_C}} \label{sec_proof_complete} Let $\upsilon$ be an AC knot. By Theorem \ref{thm_A}, $\upsilon$ is concordant to a prime satellite knot $\upsilon'$. By Theorem \ref{thm_B}, $\upsilon$ is concordant to a prime hyperbolic knot. To prove Theorem \ref{thm_C}, it is sufficient to prove that $\upsilon',\upsilon''$ may be chosen to have the same Alexander polynomial as $\upsilon$. The knot $\upsilon''$ is obtained from $\upsilon'$ by tangle operations, so first it is shown that $\upsilon'$ is AC and has the same Alexander polynomial as $\upsilon$. Then it is shown that $\upsilon''$ is AC and has the same Alexander polynomial as $\upsilon'$. First observe that a satellite of an AC knot is always AC. This is proved in the following lemma.

\begin{lemma} Let $K'$ be a satellite with pattern $P\subset D^2 \times S^1$ and companion $K\subset \Sigma \times I$. If $[K]=0 \in H_1(\Sigma\times I)$ or $[P]=0 \in H_1(D^2 \times S^1)$, then $[K']=0 \in H_1(\Sigma\times I)$. In particular, satellite virtual knots represented with either homologically trivial pattern or companion are AC.
\end{lemma}
\begin{proof} Let $N(K)$ be a regular neighborhood of $K$ and $f:D^2 \times S^1 \to N(K) \subset \Sigma \times I$ be the map defining the satellite. Let $J$ be the knot $\{0\} \times S^1$. Then $[P]=r \cdot [J] \in H_1(D^2 \times S^1)$ for some integer $r$, where $r=0$ if $P$ is homologically trivial. Now, $f_*:H_1(D^2 \times S^1) \to H_1(\Sigma\times I)$ maps $[J]$ to $\pm [K]$. This implies that $[K']=f_*([P])=\pm r[K]=0$, if either hypothesis is satisfied.
\end{proof}

Now suppose that $K \subset \Sigma \times I$ is a minimal representative of a nontrivial virtual knot $\upsilon$ and that $P$ is the pattern used in the proof of Theorem \ref{thm_A} (see Figure \ref{fig_sat_same_alex}, top left). The knot $K'$ constructed in the proof of Theorem \ref{thm_A} may be chosen as any satellite of $K$ with pattern $P$. Now, apply the skein relation of Equation (\ref{skein}) to $P$. The ball in which the relation is applied is the small green ball in $B^2 \times S^1$ depicted in Figure \ref{fig_sat_same_alex}. If the crossing is switched, the new pattern $P'$ is a knot equivalent in $B^2 \times S^1$ to $\{0\} \times S^1$ (see top of Figure \ref{fig_sat_same_alex}). The satellite knot with this pattern and companion $K$ is equivalent to $K$ in $\Sigma \times I$. On the other hand, if the crossing is smoothed, the result is a two component pattern $P_0$ in $B^2 \times I$ (see bottom of Figure \ref{fig_sat_same_alex}).  In order to show $\Delta_K(t)=\Delta_{K'}(t)$, it is sufficient to show that the satellite $L_0$ with pattern $P_0$ has Alexander polynomial $0$. By Lemma \ref{lemma_bbc}, it is sufficient to show that the satellite $L_0$ may be chosen so that $L_0$ is a boundary link. 

Let $F$ be an arbitrary Seifert surface for $K$ and let $N=N(K)$ be a regular neighborhood of $K$ such that $F \cap N$ is an annulus. Define $F'=F\smallsetminus F\cap \text{int}(N)$. Let $M\approx F' \times I$ be a regular neighborhood of $F'$ in $\Sigma \times I \smallsetminus \text{int}(N)$ that intersects $\partial N$ in a regular neighborhood of $F \cap \partial N$ in $\partial N$. Identify $F'$ with $F' \times \{1/2\}$. Then $F'\times \{\varepsilon\}$ intersects $\partial N$ in a longitude of $K$ for $0 \le \varepsilon \le 1$.  Consider the surface $F' \times \{0,1/2,1\}$ with a half twisted band attached along $F \times \{1/2,1\}$ in $N$ so that the two components of $\partial F \times \{1/2,1\}$ are joined to a single component (see Figure \ref{fig_lzero_bound}). The resulting surface $F_0'$ has two components. For an appropriately chosen map $f:B^2 \times S^1 \to \Sigma \times I$, the pattern $P_0 \to \partial F_0'$. Thus if $L_0:=f(P_0)$, it follows that $L_0$ is a boundary link in $\Sigma \times I$ and $\Delta_{L_0}(t)=0$. Setting $K'=f(P)$, we have that $K'$ is a prime satellite AC knot satisfying $\Delta_{K}(t)\doteq \Delta_{K'}(t)$. This proves the satellite case of Theorem \ref{thm_C}.

\begin{figure}[htb]
\begin{tabular}{c} \def\svgwidth{5.5in}
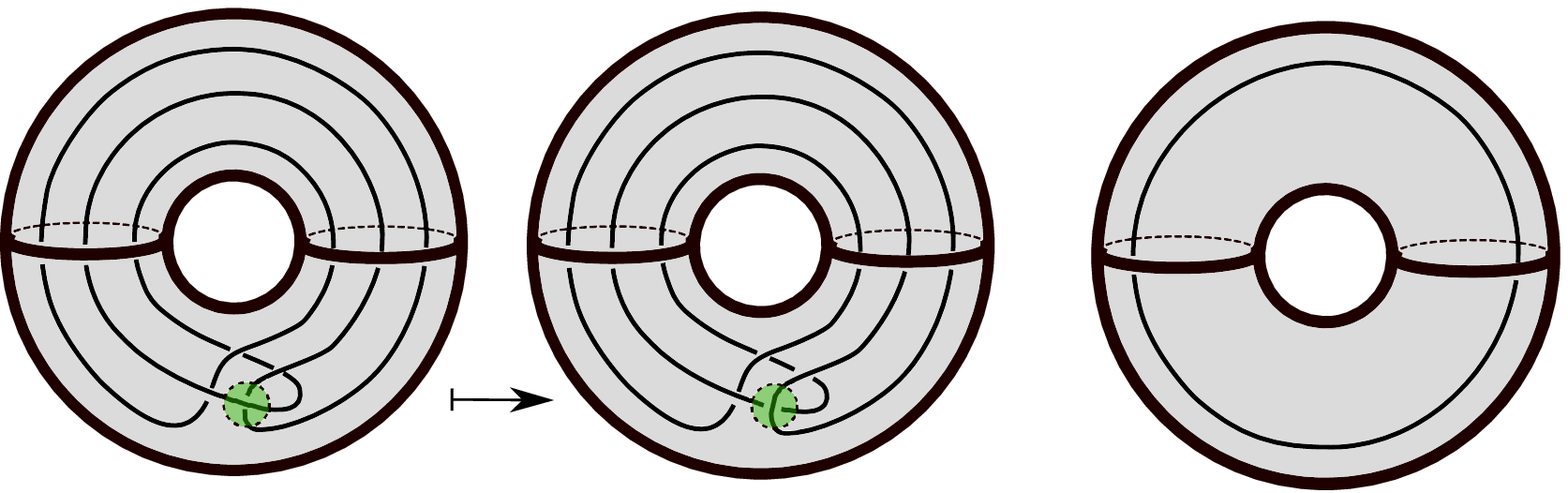\\
\def\svgwidth{5.5in} 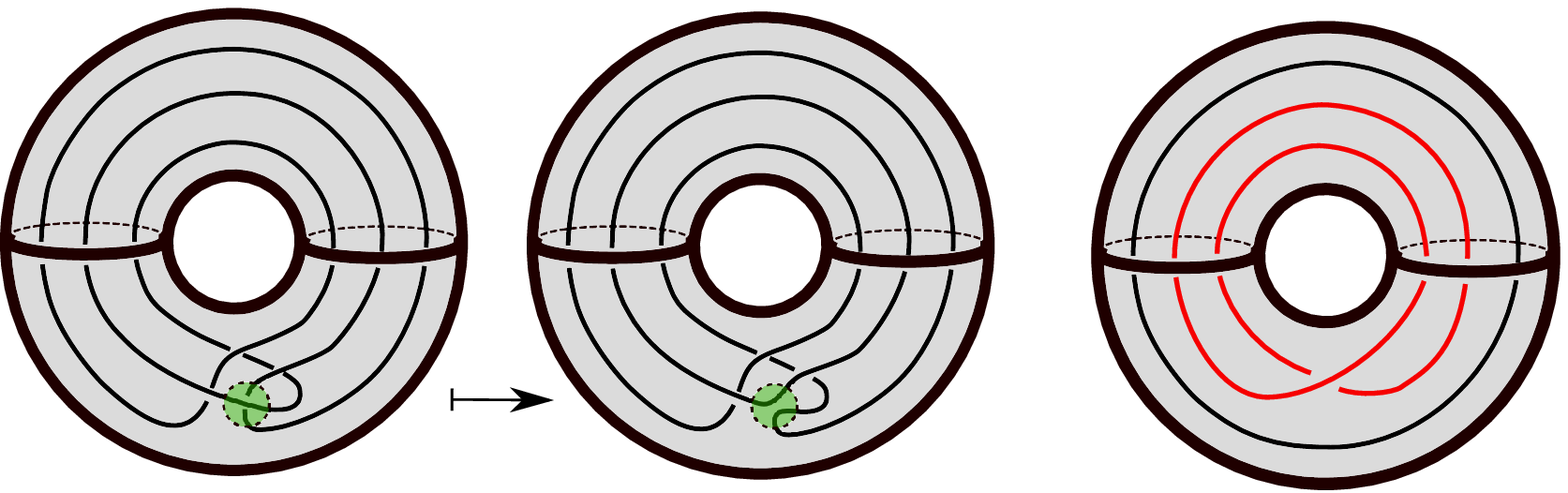 \end{tabular}
\caption{Applying the skein relation to Livingston's pattern. Switching the crossing gives a satellite knot equivalent to the companion. Smoothing the crossing gives a pattern link. A satellite with this pattern may be chosen so that it is a boundary link $L_0$ in $\Sigma \times I$ (see Figure \ref{fig_lzero_bound}).} \label{fig_sat_same_alex}
\end{figure}

\begin{remark} For classical knots, the Alexander polynomial of a satellite knot with companion $K$ and pattern $P$ is given by $\Delta_P(t)\cdot \Delta_K(t^n)$, where $n$ is the winding number of $P$ (see \cite{lick}, Theorem 6.15). H. U. Boden has informed the author that the same relation is also true for AC knots. This can be used to give an alternate proof for the satellite case of Theorem \ref{thm_C}.
\end{remark}

\begin{figure}[htb]
\begin{tabular}{c} \def\svgwidth{4in}
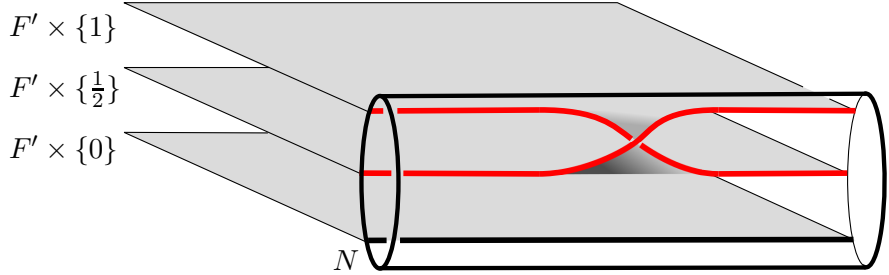 \end{tabular}
\caption{Constructing a Seifert surface for the boundary link $L_0$.} \label{fig_lzero_bound}
\end{figure}

Lastly, it must be shown that the prime hyperbolic knot $\upsilon''$ of Theorem \ref{thm_B} is AC and may be chosen to have the same Alexander polynomial as the given virtual knot $\upsilon$. Recall from the proof of Theorem \ref{thm_B} that a representative $K''\subset \Sigma \times I$ of $\upsilon''$ was given by $(K'\ominus_B t) \oplus_{B'} t'$, where $K'$ is a satellite AC knot with Livingston's pattern and companion $K$, $B$ is a $3$-ball intersecting $K'$ in an trivial 2-string tangle, and $(B',t')$ is the tangle $KT$ of Figure \ref{fig_kt}. The following lemma shows that $(K'\ominus_B t) \oplus_B' t'$ is also AC.

\begin{lemma} Suppose $J \subset \Sigma \times I$ is a homologically trivial representative of an AC knot and $(B,t)$ is any ($2$-string) tangle intersecting $J$. Let $(B',t')$ be any (2-string) tangle. Then $(J\ominus_B t) \oplus_{B'} t'$ is also AC.
\end{lemma}
\begin{proof} From Definition \ref{defn_tangle_join}, it is clear that $J$ and $(J\ominus_B t) \oplus_{B'} t'$ are homotopic embeddings of $S^1$ into $\Sigma \times I$. Thus, $J$ and $(J\ominus_B t) \oplus_{B'} t'$ represent the same homology class in $H_1(\Sigma \times I)$. This implies that $(J\ominus_B t) \oplus_{B'} t'$ is null homologous whenever $J$ is null homologous.
\end{proof}

To show that $K''=(K'\ominus_B t) \oplus_{B'} t'$ has the same Alexander polynomial of $K'$ and thus of $K$, apply the skein relation in Equation (\ref{skein}) to a crossing in the tangle $KT$. This is depicted in Figure \ref{fig_hyp_same_alex}. Switching the crossing in the green ball produces a knot in $\Sigma \times I$ that is equivalent to $K'$. Smoothing the crossing gives a two component link $L_0'$. Figure \ref{fig_kt_bound} shows a concordance of $L_0'$. The link on the the bottom right of Figure \ref{fig_kt_bound} is a boundary link in $\Sigma \times I$. Indeed, it is a split two component link $L_0''$ in $\Sigma \times I$ with one component equivalent to the AC knot $K'$ and the other component a knot bounding a disc in $\Sigma \times I$. Thus $L_0''$ is a boundary link and it follows that $\Delta_{K''}(t)\doteq\Delta_{K'}(t)\doteq\Delta_{K}(t)$.

\begin{figure}[htb]
\begin{tabular}{c} \def\svgwidth{5.5in}
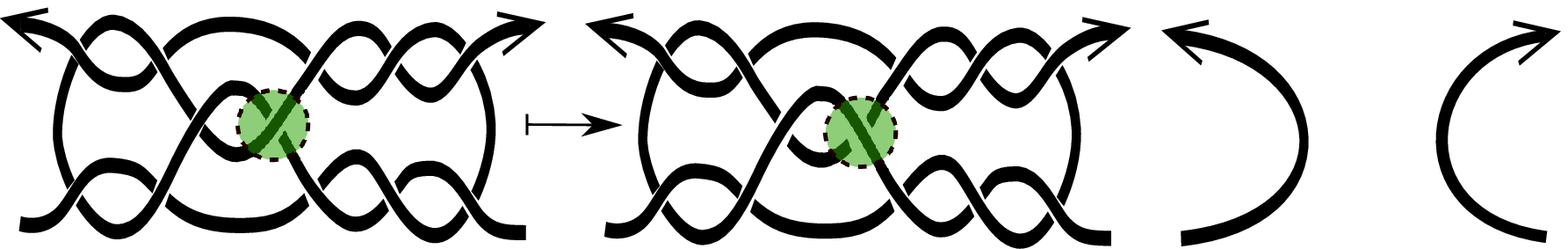\\ \\
\def\svgwidth{5.5in} 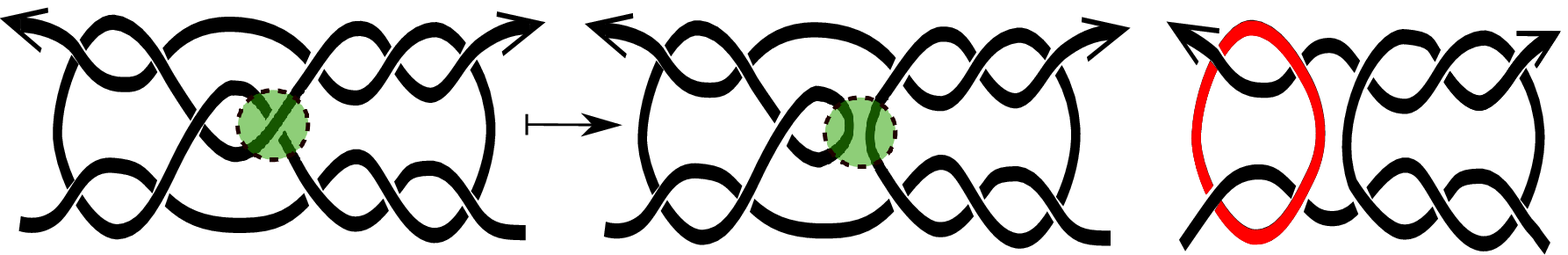 \end{tabular}
\caption{Applying the skein relation to the Kinoshita-Terasaka tangle.} \label{fig_hyp_same_alex}
\end{figure}

\begin{figure}[htb]
\begin{tabular}{c} \def\svgwidth{4.5in}
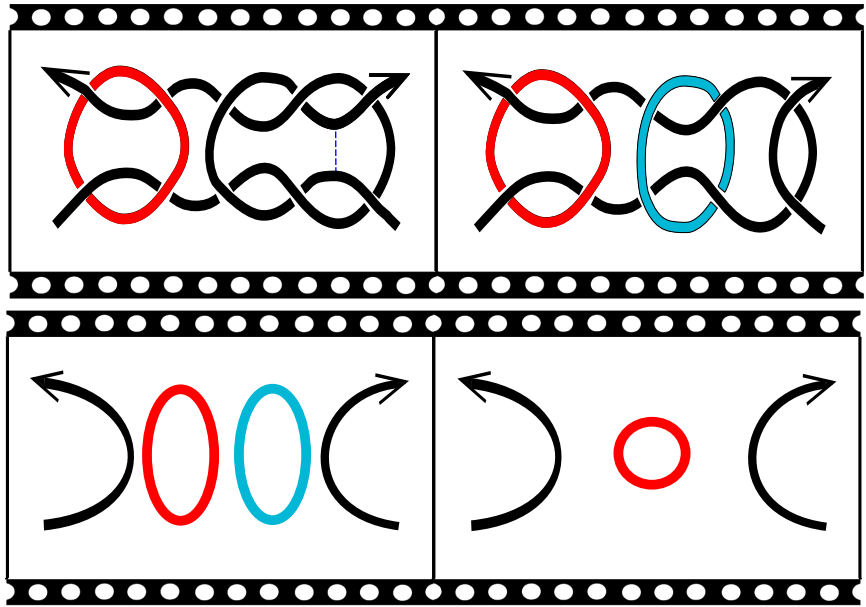 \end{tabular}
\caption{A concordance to a boundary virtual link. Apply a saddle move along the blue dotted line and then kill the blue component.} \label{fig_kt_bound}
\end{figure} 

\section{Conclusion}

We conclude with some directions for further research. Friedl-Livingston-Zentner \cite{friedl_alt} showed that there are knots in $S^3$ that are not concordant to any alternating knot. Since prime alternating non-classical virtual knots are hyperbolic, it would be useful to know whether or not every virtual concordance class contains an alternating representative. 

A possible generalization of Theorem \ref{thm_C} would be to show that every virtual knot is concordant to a prime satellite (or hyperbolic) virtual knot having the same generalized Alexander polynomial \cite{saw,silwill}. The generalized Alexander polynomial is an invariant on the full set of virtual knots and, like the Alexander polynomial of AC knots, satisfies a skein relation on the triple of diagrams $(L_+,L_-,L_0)$. However, the proof technique used in Section \ref{sec_proof_2} does not carry over. In fact, it can be shown that the knots $\upsilon$ and $\upsilon'$ in Figure \ref{fig_outline} have different generalized Alexander polynomial. 

A recent result of Myers \cite{myers_new} states that every Seifert surface (other than the disc) of a knot in $S^3$ is concordant to a Seifert surface of a hyperbolic knot having arbitrarily large volume. Friedl \cite{friedl_real} proved that every Seifert matrix of a knot in $S^3$ is $S$-equivalent to the Seifert matrix of a hyperbolic knot. To what extent do these results hold for AC knots?   

There are 1,701,936 prime classical knots up to 16 crossings (see Hoste-Thistlethwait-Weeks  \cite{hoste_first}). Of these, there are exactly 32 that are not hyperbolic. Thus, proving that a knot is hyperbolic is an effective strategy for proving that a classical knot is prime (see e.g. Thurston \cite{thurston}, Corollary 2.5). This fails for virtual knots as some hyperbolic virtual knots are not prime. There are other techniques for proving primeness for classical knots such as tangle decomposition, the knot genus, and the bridge number. Here we have used decompositions into tangles and complementary tangles to prove primeness. It would be interesting to see how often this method could be applied to prove primeness for the virtual knots in Green's table \cite{green}.

\subsection*{Acknowledgments} The author is grateful for helpful discussions with C. Adams, H. U. Boden, A. Champanerkar, I. Kofman, A. Nicas, N. Petit, and R. Todd. Thanks are also due to an anonymous reader whose comments on an earlier version of this paper led to improved exposition. This work was partially supported by funds from The Ohio State University.
\bibliographystyle{plain}
\bibliography{prime_and_sat_bib}

\end{document}